\documentclass[]{amsart}
\usepackage{amsmath,amsfonts,amssymb, amsthm}
\usepackage[abs]{overpic}
\usepackage{latexsym, graphicx, enumerate}
\usepackage{tikz-cd}
\usepackage{color}
\usepackage{url,hyperref}

\AtBeginDocument{%
   \def\MR#1{}
}

\numberwithin{equation}{section}

\theoremstyle{plain}
\newtheorem{theorem}{Theorem}
\numberwithin{theorem}{section}
\newtheorem{lemma}[theorem]{Lemma}
\newtheorem{corollary}[theorem]{Corollary}
\newtheorem{prop}[theorem]{Proposition}

\newtheorem*{namedtheorem}{\theoremname}
\newcommand{\theoremname}{testing}
\newenvironment{named}[1]{\renewcommand{\theoremname}{#1}\begin{namedtheorem}}{\end{namedtheorem}}

\theoremstyle{definition}
\newtheorem{remark}[theorem]{Remark}

\newtheorem{defn}[theorem]{Definition}

{\medskip}

\newcommand{\R}{\mathbb{R}}
\newcommand{\Z}{\mathbb{Z}}

\newcommand{\Q}{\mathbb{Q}}
\newcommand{\C}{\mathbb{C}}
\newcommand{\HH}{\mathbb{H}}
\newcommand{\orb}{{\mathcal O}} % phase this out
\newcommand{\orbO}{{\mathcal O}}
\newcommand{\omin}{{\orbO_{\min}}}
\newcommand{\orbP}{{\mathcal P}}

\newcommand{\ee}{{\mathbf e}}
\newcommand{\vv}{{\mathbf v}}
\newcommand{\xx}{{\mathbf x}}
\newcommand{\yy}{{\mathbf y}}
\newcommand{\bdy}{{\partial}}

\newcommand{\inv}{{-1}}
\newcommand{\trace}{\operatorname{tr}}
\newcommand{\isom}{\operatorname{Isom}}
\newcommand{\Comm}{\operatorname{Comm}}
\newcommand{\from}{\colon}
\newcommand{\Ext}{\operatorname{Ext}}

\renewcommand{\setminus}{\smallsetminus}

\renewcommand{\bar}[1]{\overline{#1}}

\date{\today}
\title[Arithmeticity and commensurability of links in thickened surfaces]
{Arithmeticity and commensurability \\ of links in thickened surfaces}
\author{David Futer} 
\address{
Department of Mathematics, 
Temple University,
Philadelphia, PA 19122}
\email{dfuter@temple.edu}

\author{Rose Kaplan-Kelly}
\address{
Mathematics, Statistics, and Actuarial Science Department, 
Le Moyne College,
Syracuse, NY 13214}
\email{kaplanr@lemoyne.edu}

\begin{document}
\begin{abstract}
The family of \textit{right-angled tiling links} consists of links built from regular $4$--valent tilings of constant-curvature surfaces that contain one or two types of tiles. The complements of these links admit complete hyperbolic structures and contain two totally geodesic checkerboard surfaces that meet at right angles. In this paper, we give a complete characterization of which right-angled tiling links are arithmetic, and which are pairwise commensurable. The arithmeticity classification exploits symmetry arguments and the combinatorial geometry of Coxeter polyhedra. The commensurability classification relies on identifying the canonical decompositions of the link complements, in addition to number-theoretic data from invariant trace fields.
\end{abstract}
% Keywords: tiling link, commensurability, arithmeticity, virtual link, coxeter group, geometric decomposition

\maketitle

\section{Introduction}
A pair of hyperbolic $3$--manifolds (or $3$--orbifolds) are called called \emph{commensurable} if they have isometric finite-sheeted covers.
The equivalence relation of commensurability provides a way to organize hyperbolic manifolds and gain information about properties that are well-behaved under covering maps. Accordingly, there has been considerable effort expended toward classifying certain families of manifolds, including knot and link complements, into commensurability classes. 

A cusped hyperbolic $3$--manifold $M$ is called \emph{arithmetic} if $M$ is commensurable to $\HH^3 / PSL_2(\mathcal{O}_d)$, where $\mathcal{O}_d$ is the ring of integers of $\Q(\sqrt{-d})$.
Arithmetic and non-arithmetic commensurability classes exhibit very different qualitative behavior. For instance, a non-arithmetic commensurability class has a unique minimal element, whereas an arithmetic commensurability class has infinitely many. (See Theorem~\ref{Thm:MargulisArithmeticity} for a precise statement.) Furthermore, number-theoretic tools related to arithmeticity provide useful commensurability invariants. Thus any study of the commensurability classes of a family of manifolds usually begins by identifying its arithmetic members.

In this paper, we study links in thickened surfaces $F \times I$, defined by alternating diagrams on a closed surface $F$ that correspond to tilings by regular polygons. We focus on the subset of these links whose exteriors admit a decomposition along their checkerboard surfaces into right-angled generalizations of polyhedra. By Definition~\ref{Def:RightAngledLink} and Theorem~\ref{Thm:RGCR}, these are exactly the links whose corresponding tilings use only $m$-gons and $n$-gons, with a $[m,n,m,n]$ pattern at every vertex.
Such links are called \emph{right-angled tiling links}. 

Our main results (Theorems~\ref{Thm:MainArithmeticity} and~\ref{Thm:MainCommensurability}) give a complete characterization of which right-angled tiling links are arithmetic, and which ones are pairwise commensurable. We prove:

\begin{theorem}\label{Thm:MainArithmeticity}
    The following right-angled tiling links have arithmetic exteriors:
    \begin{itemize}
    \item The links on $S^2$ corresponding to the $[3,3,3,3]$ and $[4,3,4,3]$ tilings.
    
        \item The links on $T^2$ corresponding to the $[4,4,4,4]$ and $[6,3,6,3]$ tilings.
         
        \item The links on $F = S_g$ for $g \geq 2$ corresponding to the $[6,4,6,4]$ and $[6,6,6,6]$ tilings.
    \end{itemize}
    All other right-angled tiling link exteriors are non-arithmetic.
\end{theorem}

We remark that when $F$ is a surface of genus $g \geq 2$, the hyperbolic structure on $M = (F \times I) \setminus L$ has totally geodesic boundary. We double $M$ along the boundary to obtain a finite-volume manifold $DM$, and Theorem~\ref{Thm:MainArithmeticity} characterizes exactly when $DM$ is arithmetic.

Theorem~\ref{Thm:MainArithmeticity} fits into a pattern of results showing that arithmetic $3$--manifolds are sparse among hyperbolic manifolds. For instance, Borel showed that there are only finitely 
many arithmetic hyperbolic manifolds up to any volume $V$~\cite{Borel:ArithmeticityVolume}, whereas for $V \geq 3$ the number of hyperbolic manifolds becomes infinite without the arithmetic restriction. Turning to links in $S^3$, Reid showed that the figure-8 knot is the only arithmetic knot complement in $S^3$ \cite{Reid:Figure8Knot}. 
Gehring, Maclachlan, and Martin proved that there are only 4 arithmetic two-bridge link complements~\cite{GMM:TwoGenerator}.
Baker, Goerner, and Reid showed that only 48 link complements in $S^3$ belong to a subclass of arithmetic manifolds called principal congruence manifolds~\cite{BakerGoernerReid:PrincipalCongruence}. 
In the most direct precursor of Theorem~\ref{Thm:MainArithmeticity}, Champanerkar, Kofman, and Purcell have identified the arithmetic links among links in the thickened torus corresponding to regular Euclidean tilings~\cite{CKP:Biperiodic}.

The proof of Theorem~\ref{Thm:MainArithmeticity} exploits the fact that right-angled tiling links are highly symmetric. Indeed, each (doubled) link exterior covers a Coxeter orbifold $\orbP$, whose fundamental domain is a Coxeter polyhedron $P$. (See Figures~\ref{Fig:AngleCalcs} and~\ref{Fig:DMReflectionQuotient} for a visual preview.) Accordingly, we can prove Theorem~\ref{Thm:MainArithmeticity} by employing Vinberg's characterization of arithmetic reflection groups \cite{Vinberg67}, which is reviewed in Section~\ref{Sec:Orbifolds}. In his original work~\cite{Vinberg67}, Vinberg used this criterion to completely characterize the arithmeticity of reflection groups in an $n$--dimensional simplex. More recently, Vinberg's criterion has been used to characterize the arithmeticity of 
various families of Coxeter polyhedra. 
%families of compact $4$-- and $5$--dimensional Coxeter polyhedra (Bogachev, Douba, and Raimbault \cite{BDR:InfiniteCommens}) 
These include
L\"obell polyhedra (Bogachev and Douba \cite{BogachevDouba:Lobell}) and certain higher-dimensional polyhedra (Bogachev, Kolpakov, and Kontorovich \cite{BogachevKolpakovKontorovich}).
Using commensurability invariants that arise in Vinberg's criterion, Bogachev, Douba, and Raimbault 
found that their specified families of hyperbolic polyhedra contain infinitely many commensurability classes \cite{BogachevDouba:Lobell, BDR:InfiniteCommens}. 
%However, they did not obtain a complete classification. 

One notable strength of our results is that we are able to give a complete commensurability classification of the exteriors of right-angled tiling links.

\begin{theorem}\label{Thm:MainCommensurability}
Suppose that $L \subset F \times I$ and $L' \subset F' \times I$ are right-angled tiling links. Then their exteriors (or doubled exteriors, if the surfaces are hyperbolic) are commensurable if and only if one of the following holds:
    \begin{itemize}
        \item Both $L$ and $L'$ correspond to the same $[m,n,m,n]$ tiling.
        \item Both $L$ and $L'$ correspond to  $[3,3,3,3]$, $[4,4,4,4]$, or $[6,6,6,6]$ tilings.
    \end{itemize}
\end{theorem}

Theorem~\ref{Thm:MainCommensurability} has various precursors among link complements in $S^3$.
Reid and Walsh found that every hyperbolic $2$--bridge knot complement is the only knot complement in its commensurability class~\cite{RW:Commensurability_2B_Knots}. They conjectured that a single commensurability class can contain at most three hyperbolic knots in $S^3$. Boileau, Boyer, Cebanu, and Walsh proved this conjecture for the class of knots without hidden symmetries \cite{BBCW:Commens}. Millichap and Worden found that only two hyperbolic 2-bridge link complements in $S^3$ share a commensurability class \cite{MW:Hidden_Sym_2B_Commens}. Meyer, Millichap, and Trapp~\cite{MMT:PretzelLinks}, and independently Kellerhals~\cite{Kellerhals:PolyhedralApproach}, determined when fully augmented pretzel link complements in $S^3$ are arithmetic and which of them are commensurable with each other. Moving to links in thickened surfaces, Champanerkar, Kofman, and Purcell studied the commensurability of links in the thickened torus built from regular Euclidean tilings  \cite{CKP:Biperiodic}, and proved (assuming a conjecture of Milnor) that these links fall into infinitely many commensurability classes. We emphasize that Theorem~\ref{Thm:MainCommensurability} is unconditional.

For arithmetic link exteriors, the proof of Theorem~\ref{Thm:MainCommensurability} relies on the number-theoretic data of invariant trace fields (see Definition~\ref{Def:InvariantTraceField}).
For non-arithmetic link exteriors, our primary tool
 is a criterion of Goodman, Heard, and Hodgson \cite{GHH:Commensurators} that can be paraphrased as follows: a pair of cusped, non-arithmetic $3$--manifolds are commensurable if and only if their canonical polyhedral decompositions lift to the same tiling of $\HH^3$. (See Theorem~\ref{Thm:GHH+Margulis} for a precise formulation.) Canonical decompositions are decompositions of cusped hyperbolic manifolds into convex polyhedra, dual to a Ford--Voronoi domain that encodes which cusp neighborhood is closest to a particular point.
See Definition~\ref{Def:CanonicalDecomp} for details.
In dimension $2$, these decompositions of hyperbolic manifolds were introduced in the early 20th century by Voronoi and Delaunay. In dimensions $3$ and higher, they have been extensively studied by many authors, including Epstein and Penner \cite{EP:Canonical}. 

Determining the exact canonical decompositions of an infinite family of hyperbolic manifolds is often difficult. Sakuma and Weeks found the canonical decompositions of certain highly symmetric link complements \cite{SW:Canonical}. Akiyoshi \cite{Akiyoshi:FordDomains}, Lackenby \cite{Lackenby:TorusBundles}, and Gu\'eritaud \cite{G:Thesis} have independently determined the canonical decompositions of punctured torus bundles. Akiyoshi, Sakuma, Wada, and Yamashita \cite{ASWY:FordDomains}, and independently Gu\'eritaud  \cite{G:Thesis}, have found these decompositions 
for the complements of $2$--bridge links. Gu\'eritaud and Schleimer found the canonical decompositions of long Dehn fillings of the Whitehead link \cite{GS:Canonical}, and Ham \cite{Ham:Canonical} extended the method to fillings of the Borromean rings. See also Ushijima \cite{Ushijima:Canonical} for a result in the context of manifolds with totally geodesic boundary. 
We extend the literature with the following result, whose proof forms the heart of Theorem~\ref{Thm:MainCommensurability}:

\begin{theorem}\label{Thm:MainCanonical}
    Let $L \subset F \times I$ be a right-angled tiling link with an alternating projection to a surface $F$ of genus $g \geq 2$. Let $DM$ be the double of the link exterior $M= (F \times I) \setminus L$ along its totally geodesic boundary. Then there is an equivariant choice of cusp neighborhoods in $DM$, such that the canonical polyhedral decomposition of $DM$ consists of regular ideal drums that correspond to the polygons in the tiling of $F$.
\end{theorem}

See Theorem~\ref{Thm:Canonical} for an expanded  statement, which provides more information for commensurability. 
We prove Theorem~\ref{Thm:MainCanonical} by again exploiting the high degree of symmetry in the link exterior, including the fact that $DM$ covers the Coxeter polyhedral orbifold $\orbP$ depicted in Figure~\ref{Fig:DMReflectionQuotient}. Along the way, we are able to show the satisfying result that $\orbP$ is almost always the minimal orbifold in the commensurability class of $DM$. See Corollary~\ref{Cor:MinimalOrbifold} for details.

The analogues of Theorem~\ref{Thm:MainCanonical} for tiling links on $S^2$ and $T^2$ also hold. See Proposition~\ref{Prop:S^3Arithmeticity} and Remark~\ref{Rem:S^3Canonical} for results on $S^2$. For links in a thickened torus, we are able to prove the following result, which expands the scope beyond right-angled links.

\begin{theorem}\label{Thm:CanonicalOnTorus}
    Let $L$ be a link corresponding to a $4$--valent tiling of $T^2$ by regular polygons. Then there is a natural choice of cusp neighborhoods in $M= (T^2\times I) \setminus L$ such that the canonical polyhedral decomposition of $M$ consists of regular ideal tetrahedra and regular ideal octahedra.
\end{theorem}

\noindent The proof of Theorem~\ref{Thm:CanonicalOnTorus} follows the same ideas as that of Theorem~\ref{Thm:Canonical}, again exploiting the symmetry of tiling link exteriors in a major way.

\subsection{Organization}
We begin in Section~\ref{Sec:LinksInSurface} by providing background on tiling links in thickened surfaces, the geometry of their exteriors, and the  right-angled structure of these link exteriors.
In Section~\ref{Sec:Orbifolds}, we review background on orbifolds, commensurability, and arithmeticity. We specifically review Coxeter polyhedra and the associated Gram matrices, which provide a way to decide whether a reflection orbifold is arithmetic and compute its arithmetic invariants. Sections~\ref{Sec:LinksInSurface} and~\ref{Sec:Orbifolds} do not contain any original results, but the background that they provide is crucial.

In Sections~\ref{Sec:Quotient orbifolds and arithmeticity} and~\ref{Sec:Canonical}, we focus on links in $F \times I$
 where $F$ is a surface of genus $g \geq 2$. In Section~\ref{Sec:Quotient orbifolds and arithmeticity}, we prove that the doubled exterior of such a link covers a Coxeter orbifold $\orbP$. Using Gram matrices, we then show 
that precisely two hyperbolic tilings, namely $[6,4,6,4]$ and $[6,6,6,6]$, correspond to links whose exteriors are arithmetic. We also prove that links corresponding to the same $[m,n,m,n]$ tiling have commensurable exteriors, establishing one direction of Theorem \ref{Thm:MainCommensurability}.
In Section~\ref{Sec:Canonical}, we find the canonical decomposition of the doubled link exterior, establishing Theorem~\ref{Thm:MainCanonical}. We use this to prove that commensurable link exteriors must correspond to the same $[m,n,m,n]$ tiling. This completes the proof of Theorem~\ref{Thm:MainCommensurability} for tiling links on surfaces of genus $g \geq 2$.

In Section~\ref{Sec:Lower-genus}, we consider right-angled tiling links corresponding to tilings of $T^2$ and $S^2$.
We prove Theorem~\ref{Thm:CanonicalOnTorus}, describing the canonical decompositions of links in $T^2 \times I$. We 
complete the proof of Theorem~\ref{Thm:MainArithmeticity}, and finish sorting all the right-angled tiling links into their commensurability classes, establishing Theorem~\ref{Thm:MainCommensurability}.

\subsection{Acknowledgements}

We thank Abhijit Champanerkar for sharing his ideas about canonical decompositions of semi-regular links. 
We thank Matthew Stover for numerous helpful discussions about arithmetic notions and arithmetic invariants. We thank 
Nikolay Bogachev for explaining Vinberg's arithmeticity criterion, and for his help with interpreting Gram matrices. We also thank Nikolay Bogachev and Sami Douba for their comments on an early draft of this paper. Finally, we thank the referee for their thorough feedback, which improved this paper's exposition.

The first author thanks the NSF for its support via grant DMS--2405046.

\section{Links in thickened surfaces and their exteriors}\label{Sec:LinksInSurface}

This section surveys some background material on regular tilings of constant-curvature surfaces, the associated \emph{tiling links}, and the geometry of the link exteriors. 
Throughout this work, $F$ denotes a closed, oriented surface, $L$ is a link in $F \times I$ where $I = [-1,1]$, and $\pi(L)$ denotes the projection diagram of $L$ on $F$.

 \subsection{Alternating links in thickened surfaces}

\begin{defn}\label{Def:RegularTiling}
Let $\mathbb X$ be one of the model spaces $\mathbb S^2$, $\mathbb E^2$, or $\HH^2$, endowed with its constant-curvature metric. An $n$--sided polygon $p \subset \mathbb X$ is called \emph{regular} if the symmetry group $\isom(p)$ is the full dihedral group $D_{2n}$.

 A \emph{tiling} of $\mathbb X$ is a partition of the space into convex polygons 
 with disjoint interiors, called \emph{tiles}. If two tiles intersect on their boundaries, they either share an edge and two vertices, or they share one vertex. A tiling $T$ is called \emph{regular} if all of its polygons are regular polygons. Two tilings, $T$ and $T'$, are equivalent if there exists an isometry $h \from \mathbb{X}\rightarrow \mathbb{X}$ which sends the tiles of $T$ to tiles of $T'$.
\end{defn}

For this work, we restrict our attention to \emph{$4$--valent tilings}, meaning that precisely four tiles meet at every vertex of the tiling. Following Datta and Gupta \cite{DG:Semi-regularTilings}, we denote the \emph{vertex type} of a vertex $v$ by $[a, b, c, d]$, where $a$, $b$, $c$, and $d$ are the numbers of sides of the four polygons meeting at $v$ in clockwise order. A tiling with only one vertex type is called an \emph{$[a, b, c, d]$ tiling}. 

\begin{defn}\label{Def:TilingLink}
Let $\mathbb X$ be one of $\mathbb S^2$, $\mathbb E^2$, or $\HH^2$, and let $\widetilde T$ be a regular $4$--valent tiling of $\mathbb X$. Observe that the tiles of $\widetilde T$ admit a checkerboard coloring: start by shading a single tile $t$, and then declare that a tile $t'$ is shaded if and only if a vertex-avoiding path from $t$ to $t'$ crosses an even number of edges. Now, let $G$ be  the color-preserving and orientation-preserving symmetry group of $T$, and and suppose that $G$ acts cocompactly on $\mathbb X$. If $G_0 < G$ is a torsion-free finite-index subgroup, then $F = \mathbb X/ G_0$ is a compact orientable surface that admits a checkerboard-colored tiling $T = \widetilde T / G$.

Given a checkerboard-colored regular tiling $T$ of a compact orientable surface $F$, we may construct a \emph{tiling link} $L \subset F \times (-1,1)$, as follows. The projection diagram $\pi(L)$ is exactly the $1$--skeleton of $T$. Crossing information is encoded as follows: the strands of $L$ run from undercrossings to overcrossings when we trace the boundary of a white face clockwise, or a shaded face counterclockwise. It follows immediately that $\pi(L)$ is an alternating diagram. We say that $\pi(L)$ \textit{corresponds to the tiling $T$}. 

The surface formed by connecting all shaded faces of the diagram $\pi(L)$ by inserting a half-twisted band at each crossing is called the \emph{shaded checkerboard surface} of $\pi(L)$. Performing the same construction to the white (unshaded) faces leads to the \emph{white checkerboard surface}. 
% The arcs where the two checkerboard surfaces intersect are called \emph{crossing arcs}. 
\end{defn}

\begin{remark}\label{Rem:OtherTilingLinks}
Our Definition~\ref{Def:TilingLink} of a tiling link is slightly more restrictive than that of Adams, Calderon, and Mayer~\cite[Construction 2.4]{ACM:K-uniform}. In the special case where $F$ is a Euclidean torus, our definition is also slightly more restrictive than Champanerkar, Kofman, and Purcell's notion of a \emph{semi-regular link} \cite[Definition 3.1]{CKP:Biperiodic}. The primary difference is that these authors also permit regular tilings with $3$--valent vertices. We exclude $3$--valent vertices from our definition in order to preserve a closer correspondence between the tilings and the link diagrams.

 Links in the thickened torus corresponding to Euclidean tilings have also been referred to as \emph{textile links} \cite{BMOP:Textile}. 
%    These families of links fall under the class of \emph{cellular weakly generalized alternating links} as defined by Howie and Purcell in \cite{HP:GeneralizedAlternating}. 
\end{remark}

\begin{defn}\label{Def:LinkExterior}
Let $F$ be a closed orientable surface, and let $L \subset F \times I$ be a link embedded in the interior of $F \times I$, where $I = [-1,1]$. The \emph{link exterior} $\Ext(L)$ is defined as follows, depending on the geometric type of $F$:
\begin{itemize}
\item If $\chi(F) > 0$, then $F = S^2$. Collapsing $F \times \{1\}$ to a point and $F \times \{-1\}$ to a point produces $S^3$. We define $\Ext(L)$ to be the complement of the resulting link in $S^3$.
\item If $\chi(F) = 0$, then $F = T^2$. We remove the boundary surfaces $F \times \{\pm 1\}$, and define $\Ext(L) = (F \times (-1, 1)) \setminus L$.
\item If $\chi(F) < 0$, then $F$ is a hyperbolic surface. We define $\Ext(L) = (F \times [-1, 1]) \setminus L$.
\end{itemize}
Note that in all cases, $\Ext(L)$ is a noncompact $3$--manifold, whose noncompact ends are homeomorphic to $T^2 \times [0, \infty)$.
\end{defn}

We are specifically interested in hyperbolic structures on link exteriors.

\begin{defn}\label{Def:HyperbolicLink}
    Let $F$ be a closed orientable surface, and let $L$ be a link embedded in the interior of $F \times [-1, 1]$. A link exterior $\Ext(L)$ is called \emph{hyperbolic} if it admits a complete hyperbolic metric, such that $\bdy \Ext(L)$ is either empty or totally geodesic. By Mostow--Prasad rigidity, such a metric is unique up to isometry. We say that $L$ is a \emph{hyperbolic link} if $\Ext(L)$ is hyperbolic.
\end{defn}

It is known that every tiling link in $F \times I$ has hyperbolic exterior. This follows from the work of 
Menasco~\cite{Menasco:Incompressible} if $F = S^2$, 
and from the work of Adams, Albors-Rivera, Haddock, Li, Nishida, and Wang~\cite[Theorems 1 and 2]{Adams_etal:Hyperbolicity} if $F$ is a surface of genus $g \geq 1$. Howie and Purcell showed that a broader class of generalized alternating link complements are hyperbolic \cite[Theorem 4.2]{HP:GeneralizedAlternating}. All of the above proofs rely on Thurston's hyperbolization theorem for Haken $3$--manifolds. The paper of Adams, Calderon, and Mayer~\cite{ACM:K-uniform} contains a direct construction of the hyperbolic metrics on all tiling link exteriors via generalized bipyramids (Definition~\ref{Def:Bipyramid}). For tiling link exteriors in $T^2 \times I$, the paper of Champanerkar, Kofman, and Purcell~\cite{CKP:Biperiodic} also contains a direct construction of hyperbolic metrics using ideal triangulations.

\begin{defn}\label{Def:RightAngledLink}
A tiling link $L \subset F \times I$ corresponding to a tiling $T$ is called \emph{right-angled} if $T$ has only one vertex type, with vertex pattern $[m,n,m,n]$. The numbers $m$ and $n$ might coincide.
\end{defn}

\begin{remark}\label{Rem:Uniformization}
If $F$ is a compact orientable surface, without a specified metric, and $T$ is a tiling of $F$ by topological polygons such that every vertex has type $[m,n,m,n]$, then there is a spherical, Euclidean, or hyperbolic metric on $F$ that makes $T$ into a regular tiling. See Edmunds, Ewing, and Kulkarni~\cite{EEK:Tessellations} and Datta and Gupta~\cite{DG:Semi-regularTilings}. Thus, in constructing a right-angled tiling link, it is not necessary to start with a \emph{regular} tiling.
\end{remark}

While Definition~\ref{Def:RightAngledLink} makes no reference to angles, the term \emph{right-angled} is motivated by $3$--dimensional hyperbolic geometry. By Theorem~\ref{Thm:RGCR}, a tiling link is right-angled in the sense of Definition~\ref{Def:RightAngledLink} if and only if $\Ext(L)$ contains two totally geodesic surfaces that meet at right angles.

\subsection{Polyhedral decompositions of link exteriors}
The process of using alternating diagrams in $S^2$ to build polyhedral decompositions of link complements in $S^3$ began in the work of Thurston~\cite[Chapter 3]{Thurston:Notes} and Menasco~\cite{Menasco:Polyhedra}. This work has been extended to link exteriors in thickened surfaces by Adams, Calderon, and Mayer~\cite{ACM:K-uniform}, Champanerkar, Kofman, and Purcell~\cite{CKP:Biperiodic}, and Howie and Purcell~\cite{HP:GeneralizedAlternating}.

\begin{defn}\label{Def:GeneralizedPolyhedron}
A \emph{(combinatorial) polyhedron} is a $3$--ball $P$, endowed with a cellular graph $\Gamma$ on its boundary. Each vertex of $\Gamma$ is declared \emph{finite}, \emph{ideal}, or \emph{ultra-ideal}.
For an ideal vertex $v \in \Gamma$, we remove $v$ from $P$. For an ultra-ideal vertex of $\Gamma$, we remove a regular neighborhood of $v$ from $P$, effectively truncating $v$.
\end{defn}

To work with hyperbolic structures on link exteriors, we endow the polyhedra themselves with geometric shapes.

\begin{figure}
\centering   
\begin{overpic}[abs,abs,unit=1mm,scale=.5, trim={0 4.75in 0 0}, clip]{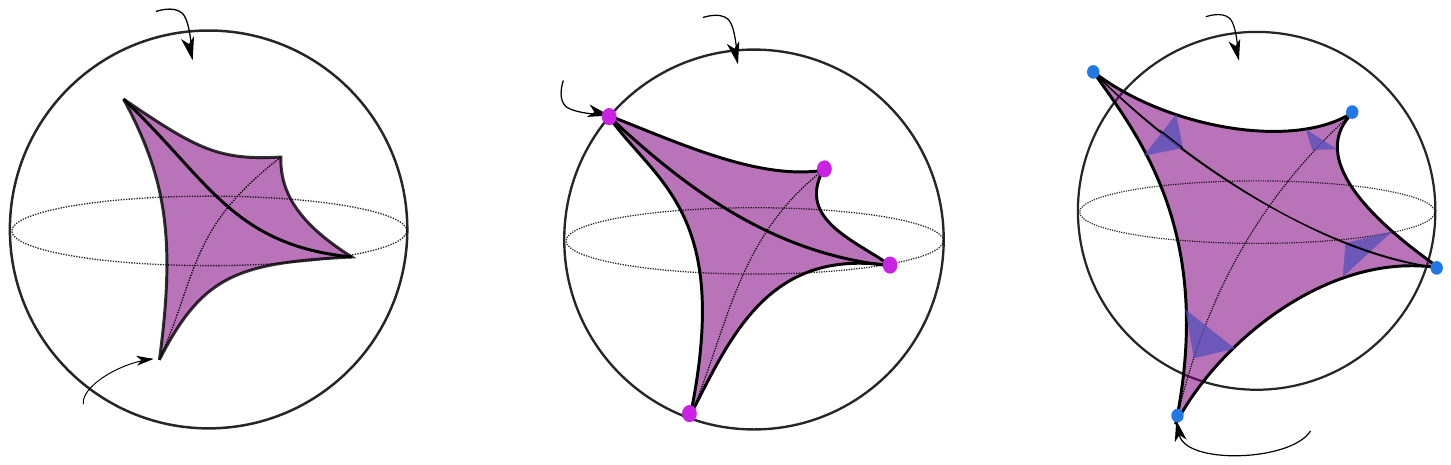}
    \put(18,38){$\partial \HH^3$}
     \put(64,38){$\partial \HH^3$}
      \put(107,38){$\partial \HH^3$}
      \put(123,4){Ultra-ideal}
      \put(14,3){Finite}
      \put(55, 35){Ideal}    
    \end{overpic}
    \caption{Hyperbolic tetrahedra with all finite, ideal, or ultra-ideal vertices. The truncation faces of ultra-ideal vertices are shown in dark purple.}
    \label{TypesofTetinH3}
\end{figure}

\begin{defn}
Let $P$ be a combinatorial polyhedron, whose vertices have been labeled as finite, ideal, or hyper-ideal. 
We say that $P$ is \emph{hyperbolic} if it admits a convex hyperbolic structure with totally geodesic faces, so that finite vertices lie within $\HH^3$, ideal vertices lie on $\bdy \HH^3$, and ultra-ideal vertices lie outside  $\bdy \HH^3$. See Figure \ref{TypesofTetinH3}.
\end{defn}

The vertices of a hyperbolic polyhedron are characterized by their links. The link of a finite vertex $v$ is a spherical polygon lying on an $\epsilon$--sphere about $v$. The link of an ideal vertex $v$ is a Euclidean polygon lying on a horosphere orthogonal to all the edges into $v$. The link of an ultra-ideal vertex is a hyperbolic polygon lying on a hyperbolic plane orthogonal to all the edges into $v$. 
This polygon is called a \emph{truncation face}.

\begin{defn}\label{Def:GeometricDecomp}
Let $F$ be a closed orientable surface, and let $L$ be a hyperbolic link in $F \times [-1, 1]$. A \emph{geometric decomposition} of $\Ext(L)$ is a decomposition into hyperbolic polyhedra $P_1, \ldots, P_n$, such that the $P_i$ are glued by isometries along their faces to form the complete hyperbolic structure on $\Ext(L)$.
\end{defn}

\begin{defn}\label{Def:Bipyramid}
A polyhedron $P \subset \HH^3$ with vertices $v_0, \ldots, v_n$ is called an \emph{$n$--pyramid} if vertex $v_0$ (called the \emph{apex}) is joined by an edge to every other $v_i$, and there is a single face (called the \emph{base}) with vertices $v_1, \ldots, v_n$. The pyramid $P$ is called \emph{regular} if the symmetry group $\isom(P)$ acts on the base as the full dihedral group $D_{2n}$.
We require $v_1, \ldots, v_n$ to be ideal, while the apex $v_0$ may be finite, ideal, or ultra-ideal. When $v_0$ is an ultra-ideal vertex, we emphasize this fact by calling $P$ a \emph{generalized regular pyramid}.

A polyhedron $\Delta \subset \HH^3$ is called a \emph{regular $n$--bipyramid} if $\Delta$ is formed by gluing two regular $n$--pyramids by an isometry of their bases.  The quotient of the two bases is a regular ideal $n$-gon, called the \emph{horizontal midsection} of $\Delta$.
The bipyramid $\Delta$ has two apexes (which may be finite, ideal, or ultra-ideal). The faces of $\Delta$ are called \emph{vertical faces}, for contrast with the horizontal midsection. By construction, a regular $n$--bipyramid is invariant by a rotation of order $n$ about an axis connecting the two apexes; this axis is called the \emph{stellating edge} of $\Delta$. See Figure \ref{Fig:Wedge}.
\end{defn}

\begin{figure}[h]
    \centering
    \begin{overpic}[abs,unit=1mm,scale=.8, trim={0 6.25in 4in 0}, clip]{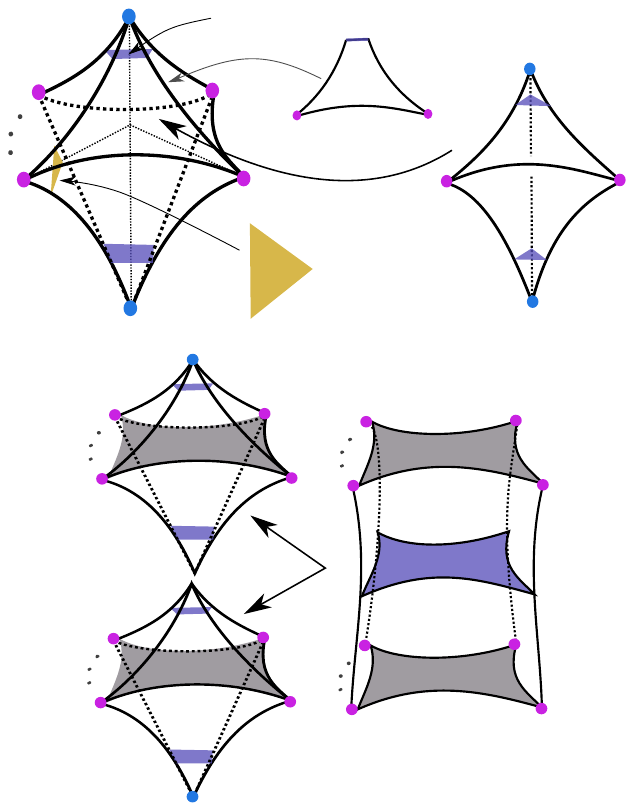}
    \put(27,109){Truncation face}
    \put(41,91){Vertical face}
    \put(35,80){$\alpha_n/2$}
     \put(35,64){$\alpha_n/2$}
     \put(43,72){$\pi-\alpha_n$}
     \put(56,11){Drum}
     \put(68,65){Wedge}
    \end{overpic}
    \caption{Above: A generalized bipyramid $\Delta$ built on a regular polygon with interior angle $\alpha_n$ and $n$ sides. Truncation faces for a neighborhood of the apexes of this bipyramid are shown in dark blue. Half of a Euclidean truncation face for one of the ideal vertices of the midsection is shown in yellow. The bipyramid $\Delta$ decomposes into $n$ tetrahedra, called \emph{wedges}, glued along the stellating edge. Below: A single $n$-drum constructed from two half-bipyramids and thus from $2n$ half-wedges.}
    \label{Fig:Wedge}
\end{figure}

Adams, Calderon, and Mayer proved that if $L \subset F \times I$ is a 
tiling link, the link exterior $\Ext(L)$  admits a geometric decomposition into regular hyperbolic bipyramids, with one $n$--bipyramid for every $n$-gon in the tiling \cite[Theorem 4.4]{ACM:K-uniform}. (In their terminology, \emph{regular} bipyramids are called \emph{symmetric}.) The apexes of each bipyramid are finite when $\chi(F)>0$, ideal when $\chi(F)=0$, and ultra-ideal when $\chi(F) < 0$. The horizontal midsection of each $n$--bipyramid corresponds to an $n$-gon face in the checkerboard coloring of the projection $\pi(L)$, and to an $n$-gon in the tiling of $F$. See \cite[Figure 11]{ACM:K-uniform} for an illustration of the gluing pattern.
When $F = T^2$, this result is also a special case of a theorem of Champanerkar, Kofman, and Purcell \cite[Theorem 3.5]{CKP:Biperiodic}.

When $F$ is a hyperbolic surface, it will be convenient to double $\Ext(L)$ along its nonempty, totally geodesic boundary. If $M = \Ext(L)$, the doubled manifold $DM$ is a finite-volume hyperbolic manifold with cusps. The doubled link exterior $DM$ can be decomposed into drums.

\begin{defn}\label{Def:Drum}
An \emph{$n$-drum} is a polyhedron $P$ with the combinatorics of $p \times [0,1]$, where $p$ is an $n$--gon. The faces $p \times \{0\}$ and $p \times \{1\}$ are called \emph{bases}, and the remaining $n$ faces are called \emph{lateral faces}. A \emph{regular ideal drum} is a hyperbolic drum $P \subset \HH^3$, with totally geodesic faces, all of whose $2n$ vertices are ideal,
which possesses all of the symmetries of a regular Euclidean prism $p \times [0,1]$.
% and such that the subgroup $G < \isom(P)$ that sends bases to bases has order $4n$.  That is, $\isom(P)$ contains a $D_{2n}$ subgroup acting dihedrally on each base and also contains a reflection interchanging the two bases.
(In the special case where $n = 4$, an $n$--drum $P$ is combinatorially a cube. 
We do not require a regular ideal $4$--drum to have symmetries that send a base to one of the lateral faces.)
\end{defn}

Decompositions of link complements into regular ideal drums appear in the work of Thurston~\cite[Section 6.8]{Thurston:Notes} and of Sakuma and Weeks~\cite{SW:Canonical}. Examples are also shown in \cite[Chapter 5]{RKK:Thesis}. In our setting, we have:

\begin{lemma}\label{Lem:Drums}
Let $F$ be a closed surface of genus $g \geq 2$, and let $L$ be a $k$--regular tiling link in $F \times I$.
Then the double of $\Ext(L)$ admits a geometric decomposition into regular ideal drums.
\end{lemma}
     
   \begin{proof}
By Adams, Calderon, and Mayer~\cite[Theorem 4.4]{ACM:K-uniform}, $\Ext(L)$ has a decomposition into regular hyperbolic (generalized) bipyramids. The truncation faces of the bipyramids tile the totally geodesic boundary of $\Ext(L)$. If we cut each bipyramid $\Delta$ along its horizontal midsection into two regular pyramids, the double of each regular (generalized) pyramid becomes an ideal drum $P$.        The result is a collection of ideal drums with faces glued according to the gluing pattern of the original bipyramids. 
Note that every drum $P$ is regular because the symmetry group of each generalized pyramid is already maximal, and the symmetry group of $P$ also interchanges the two constituent generalized pyramids.
   \end{proof}

Observe that each regular ideal $n$--drum $P$ from Lemma~\ref{Lem:Drums} is built from reassembled pieces of $n$ wedges because each $n$-drum consists of half of two $n$-bipyramids. See Figure~\ref{Fig:Wedge}. Since the base-preserving symmetry group $G < \isom(P)$ has order $4n$ by Definition~\ref{Def:Drum}, the quotient $P/G$ is the quotient of a wedge by two reflections.

\subsection{Right-angled structure on link exteriors in thickened surfaces}
The following result, due to Gan~\cite{Gan:Alternating} when $F = S^2$ and to Kaplan-Kelly~\cite[Theorem 6]{RKK:RightAngled} when $F$ has genus $g \geq 1$, explains the geometric meaning behind the term \emph{right-angled tiling link}.

\begin{theorem}% (\cite[Theorem 6]{RKK:RightAngled}).
\label{Thm:RGCR}
Let $F$ be a closed surface of genus $g \geq 0$, and let $L \subset F \times I$ be a tiling link with projection diagram $\pi(L)$.
Then the following are equivalent:
    
    \begin{enumerate}
        \item $L$ is a right-angled tiling link, as in Definition~\ref{Def:RightAngledLink}.
        \item Each checkerboard surface of $\pi(L)$ has exactly one type of polygon.
        \item The checkerboard surfaces of $\pi(L)$ are totally geodesic.
        \item\label{Itm:RGCR} 
The checkerboard surfaces of $\pi(L)$ are totally geodesic, and meet at right angles.
    \end{enumerate}
\end{theorem}

We remark that \cite[Theorem 6]{RKK:RightAngled} applies to links that are slightly more general than tiling links. In the work of Kaplan-Kelly~\cite{RKK:RightAngled,RKK:Thesis}, the links satisfying condition~\eqref{Itm:RGCR} are called \emph{RGCR links}.

\section{Orbifolds, commensurability, and arithmeticity}
\label{Sec:Orbifolds}

This section provides background on orbifolds, commensurability, and arithmetic notions. While none of the results surveyed here are original, we found it useful to collect the definitions and some important theorems in one place.

We begin by reviewing Kleinian groups. A \emph{Kleinian group} is a discrete subgroup $\Gamma < \isom(\HH^3)$. In our usage, a Kleinian group $\Gamma$ is not required to be contained in the orientation-preserving subgroup $\isom^+(\HH^3) \cong PSL(2,\C)$.
If a Kleinian group $\Gamma$ is torsion-free, then the quotient $M = \HH^3 / \Gamma$ is a hyperbolic $3$--manifold. If $\Gamma$ contains torsion, then its action contains fixed points, hence the quotient $\HH^3 / \Gamma$ is an example of a geometric $3$--orbifold, defined in Definition~\ref{Def:Orbifold}.

\subsection{Orbifolds}
Here we recall information about orbifolds, their singularities, and their notation. See Cooper, Hodgson, and Kerckhoff \cite{CHK:Orbifolds} or Walsh \cite{Walsh:Survey} for excellent references.

\begin{defn}\label{Def:Orbifold}
 A \emph{geometric $n$--orbifold} $\orbO$ is the quotient of simply connected Riemannian $n$--manifold $M$ by a discrete subgroup  $\Gamma < \isom(M)$. The geometric orbifold $\orbO$ is called \emph{orientable} if $\Gamma$ is orientation-preserving, and \emph{non-orientable} otherwise. 
    
    The path-metric on $M$ descends to a path-metric on $\orbO$. The group $\Gamma$ is called the \emph{orbifold fundamental group} of $\orbO$, and denoted $\pi_1(\orbO)$. We write $M = \widetilde \orbO$. The universal covering map $f: M \to \orbO$ is a local isometry with respect to the path metric. 
\end{defn}

There is a more general theory of orbifolds, without a given metric structure and without the assumption that the orbifold is covered by an $n$--manifold. See \cite[Chapter 2]{CHK:Orbifolds} for an introduction. Orbifolds that are covered by manifolds, as in Definition~\ref{Def:Orbifold}, are called \emph{good}. 
In the sequel, the word ``geometric'' will be presumed when describing orbifolds. Furthermore, in most applications, the universal cover $M = \widetilde \orbO$ will carry a constant-curvature metric.

\begin{defn}\label{Def:SingularLocus}
Let $\orbO = M / \Gamma$ be an $n$--orbifold as in Definition~\ref{Def:Orbifold}.
 For a point $p \in M$, let $\Gamma_p < \Gamma$ be the stabilizer of $p$. If $\Gamma_p \neq \{1\}$, the quotient point $x = f(p) \in \orbO$ is called a \emph{singular point} of  $ \orbO$. For a sufficiently small $\epsilon > 0$, the metric ball $B_\epsilon(x)$ is isometric to $B_\epsilon(p)/\Gamma_p$. The (necessarily finite) group $\Gamma_p$ is called the \emph{local group} of $x$. Note that the choice of preimage  $p \in f^\inv (x)$ only changes the local group by conjugation in $\isom(M)$.

The \emph{singular locus} of $\orbO$, denoted $\Sigma_\orbO$, is the set of all singular points. The \emph{mirror locus} of $\orbO$, denoted $\Sigma_\orb^{\mathrm{mir}}$, consists of singular points $x \in \orbO$ whose point group $\Gamma_p$ contains a reflection in some codimension--1 hyperplane $H \subset M$. In the special case where $\Gamma_p = \{1, \gamma\}$ and $\gamma$ is a reflection in $H$, the metric ball $B_\epsilon(x) = B_\epsilon(p)/\Gamma_p$ is modeled on a half-ball in $\R^n$ with a ``mirror'' running along the quotient of $H$. Note that the mirror points of $\orbO$ are not boundary points, because geodesics in $B_\epsilon(x)$ bounce off the mirror locus and continue.
\end{defn}

% One consequence of Definition~\ref{Def:SingularLocus} is that if $\Gamma$ does not have any reflections, hence $\orbO$ does not have any mirror locus, then the singular locus $\Sigma_\orbO$ is a cell complex of codimension at least $2$. In particular, if $\orbO$ is $2$--dimensional and has no mirror locus, then every point group $\Gamma_p$ is a cyclic group $\Z_r \cong \langle e^{2 \pi i /r} \rangle$ acting by rotations about $p$. It follows that the underlying topological space of $\orbO$ is homeomorphic to a surface $X^2$, and $\Sigma_\orbO$ is a discrete set of points, called \emph{cone points}. If the surface $X^2$ has finite type, the orbifold $\orbO$ is classified up to equivalence by the data $X^2(r_1, \ldots, r_n)$, where $r_1, \ldots, r_n$ are the orders of the point groups at the cone points. Here, equivalence means that orbifolds $\orbO = M / \Gamma$ and $\orb' = M' / \Gamma'$ have universal covers related by a diffeomorphism $M \to M'$ that carries $\Gamma$ to $\Gamma'$.

% In the special case where the underlying space of $\orbO$ is $S^2$, the orbifold $S^2(\ell,m,n)$ is called a \emph{turnover} and the orbifold $S^2(2,2,2,2)$ is called a \emph{pillowcase}.  Turnovers admit a hyperbolic structure whenever $\frac{1}{\ell}+\frac{1}{m}+\frac{1}{n} < 1$. Pillowcases admit a Euclidean structure. See \cite[Theorem 2.22]{CHK:Orbifolds}.

\begin{defn}
Definitions~\ref{Def:Orbifold} and~\ref{Def:SingularLocus} extend naturally to the setting where $M$ is a simply connected $n$--manifold with boundary. In this case, $\orbO = M / \Gamma$ is an $n$--orbifold with boundary, and $\bdy \orbO$ is a $(n-1)$--orbifold. As mentioned above, the mirror locus $\Sigma_\orb^{\mathrm{mir}}$ does not consist of boundary points (although it may intersect $\bdy \orbO$ transversely).
\end{defn}

\begin{defn}\label{Def:Silvering}
Let $\orbO$ be an $n$--orbifold with boundary. Assume that $S = \bdy \orbO$ is non-empty and totally geodesic. Let $D\orbO$ denote the double of $\orbO$ along $\bdy \orbO$. By construction, $S = \bdy \orbO$ embeds into $D \orbO$ as a totally geodesic hypersurface, invariant under a reflection $\gamma$ interchanging the two copies of $\orbO$. The quotient $\orbO^\pm = D \orbO / \langle \gamma \rangle$ is an orbifold homeomorphic to $\orbO$, with the important difference that the hypersurface $S$ has become mirror locus in $\orb^\pm$. The process of turning $S$ from boundary into mirror locus is called \emph{silvering}. See \cite[Section 2.1]{CHK:Orbifolds}.

By construction, the doubled orbifold $D \orbO$ and the silvered orbifold $\orb^\pm$ have empty boundary. They share a common universal cover $\widetilde{M^\pm} = \widetilde{\orb^\pm}$, which can be obtained from $M = \widetilde \orbO$ by iterated reflection along the totally geodesic boundary $\bdy M$.
\end{defn}

The following special case will be of primary interest below (see Corollary~\ref{Cor:SameTileCommens}). If $\orbO$ is a hyperbolic $3$--orbifold with totally geodesic boundary, the universal cover $M = \widetilde \orbO$ is a subset of $\HH^3$ cut out by some number of disjoint hyperbolic planes. In this case, $\widetilde{\orb^\pm}$ will be all of $\HH^3$.

\subsection{Commensurability and arithmeticity}
Next, we review some background on commensurability and arithmeticity. We recommend Maclachlan and Reid \cite{MR:Arithmetic} as an excellent reference.

\begin{defn}\label{Def:Commensurable}
Let $M\cong \HH^3/ \Gamma$ and $M'\cong \HH^3/ \Gamma'$ be hyperbolic $3$--orbifolds. We say that $M$ and $M'$ are \emph{commensurable} if they have isometric finite-sheeted covers. 
The Kleinian groups $\Gamma$ and $\Gamma'$ are commensurable if $\Gamma$ and a conjugate of $\Gamma'$ have a common finite index subgroup. 

The \emph{commensurator} of $\Gamma$ is defined to be
\[
\Comm(\Gamma)=\{g \in \isom(\HH^3) \: | \: [\Gamma : \Gamma \cap g\Gamma g^{-1}]< \infty\},
\]
and similarly for $\Gamma'$. It follows that $\Gamma$ and $\Gamma'$ are commensurable if and only if their commensurators coincide, up to conjugacy. Compare \cite[Lemma 2.3]{Walsh:Survey}.
\end{defn}

Generalizing Definition~\ref{Def:Commensurable}, suppose that  $M$ and $M'$ are hyperbolic $3$--manifolds with totally geodesic boundary. We say that $M$ and $M'$ are commensurable if they have isometric finite-sheeted covers. This generalized notion of commensurability will be referenced in Corollary~\ref{Cor:SameTileCommens}.

\begin{remark}\label{Rem:EasyCommensurability}
If finite-volume hyperbolic $3$--manifolds $M$ and $M'$ cover a common $3$--orbifold $\orbO$, then they must be commensurable. The finite-index cover of $M$ and $M'$ corresponds to the subgroup $\pi_1(M) \cap \pi_1(M') < \pi_1(\orbO)$. 

The converse of this statement only holds for non-arithmetic manifolds; see Theorem~\ref{Thm:MargulisArithmeticity}.  
\end{remark}

% The following definition is stated in a simple form that applies only to $3$--manifolds with cusps.

\begin{defn}\label{Def:Arithmetic}
Let $d$ be a positive, square-free integer. The \emph{Bianchi group} determined by $d$ is $PSL(2, \mathcal{O}_d)$, where $\mathcal{O}_d$ is the ring of integers of $\Q(\sqrt{-d})$. This group is discrete in $PSL(2,\C)$, and the quotient $\HH^3 / PSL(2, \mathcal{O}_d)$ is a cusped, orientable hyperbolic $3$--orbifold of finite volume. Since $PSL(2, \mathcal{O}_d)$ contains parabolic elements -- for instance, $[ \begin{smallmatrix}1 & 1 \\0 & 1 \end{smallmatrix} ]$ -- the quotient orbifold $\orbO$ contains cusps.
See \cite[Section 1.4.1]{MR:Arithmetic} for more on Bianchi groups.

A Kleinian group $\Gamma$ containing parabolics is called \emph{arithmetic} if $\Gamma$ is  commensurable to the Bianchi group $PSL_2(\mathcal{O}_d)$ for some $d$. In this situation, the quotient manifold or orbifold $M = \HH^3 / \Gamma$ is also called arithmetic.
\end{defn}

The full definition of arithmetic Kleinian groups, which also covers cocompact groups, is somewhat more involved. See \cite[Chapter 8]{MR:Arithmetic} for details. For Kleinian groups with parabolics, Definition~\ref{Def:Arithmetic} is equivalent to the general definition by  \cite[Theorem 8.2.3]{MR:Arithmetic}.

The following fundamental theorem of Margulis \cite{Margulis:Discrete} characterizes arithmeticity via commensurators. See also Borel~\cite{Borel:ArithmeticityVolume} for an account tailored to Kleinian groups.

\begin{theorem}[Margulis]
\label{Thm:MargulisArithmeticity}
Let $M = \HH^3 / \Gamma$ be a finite volume hyperbolic $3$--manifold.
\begin{enumerate}[\: \: $(1)$]
\item\label{Itm:MargulisAr} If $M$ is arithmetic, then $\Comm(\Gamma)$ is dense in $\operatorname{Isom}(\HH^3)$. The commensurability class of $M$ contains infinitely many mimimal elements.
\item\label{Itm:MargulisNonAr} If $M$ is non-arithmetic, then $\Comm(\Gamma)$ is discrete in $\operatorname{Isom}(\HH^3)$.  Hence $\orbO = \HH^3 /\Comm(\Gamma)$ is the unique minimal $3$--orbifold covered by every orbifold or manifold commensurable to $M$.
\end{enumerate}
\end{theorem}

One consequence of Theorem~\ref{Thm:MargulisArithmeticity} is that non-arithmetic finite-volume hyperbolic $3$-manifolds $M$ and $M'$ are commensurable if and only if they cover a common quotient orbifold, namely $\orbO \cong \HH^3 / \Comm(\Gamma) \cong \HH^3 / \Comm(\Gamma')$. This provides a converse to Remark~\ref{Rem:EasyCommensurability}.

\begin{defn}\label{Def:InvariantTraceField}
Let $\Gamma$ be a Kleinian group such that $\HH^3/\Gamma$ is orientable and has finite volume. Let $\Gamma^{(2)}= \langle \gamma^2 : \gamma \in \Gamma \rangle$. The \emph{invariant trace field} of $\Gamma$ is the field $k\Gamma = \Q(\operatorname{tr} \Gamma^{(2)})$. This is a number field, meaning  a finite extension of $\Q$ \cite[Theorem 3.1.2]{MR:Arithmetic}.
\end{defn}

It is known that $k\Gamma$ is an invariant of the commensurability class of $\Gamma$ \cite[Theorem 3.34]{MR:Arithmetic}. For example, if $\Gamma$ is commensurable to the Bianchi group $PSL_2(\mathcal{O}_d)$, then $k\Gamma = \Q(\sqrt{d})$. Compare \cite[Section 4.1]{MR:Arithmetic}.

\subsection{Coxeter polyhedra and reflection groups}
Next, we review (Kleinian) Coxeter groups and the associated orbifolds. This special class of hyperbolic orbifolds plays an important role in our arguments. 

\begin{defn}\label{Def:CoxeterPoly}
Let $P \subset \HH^3$ be a convex hyperbolic polyhedron with totally geodesic faces $F_1, \ldots, F_s$. The polyhedron $P$ is permitted to have ideal and ultra-ideal vertices, as in Figure~\ref{TypesofTetinH3}. We call $P$ a \emph{Coxeter polyhedron} if all of its dihedral angles are of the form $\pi/n$, where $n \geq 2$ is an integer.

If $P \subset \HH^3$ is a Coxeter polyhedron, the \emph{Coxeter group} associated to $P$ is the Kleinian group  $\Gamma(P) < \isom(\HH^3)$ generated by reflections in the faces of $P$. This group is always discrete, and has $P$ as a fundamental domain. Thus, if $P$ has finite volume, then $\Gamma(P)$ has finite co-volume.
The quotient orbifold $\orbP = \HH^3 / \Gamma(P)$ has underlying space $P$, and mirror locus along the faces of $P$. 

The Kleinian group $\Gamma^+(P) < \Gamma(P)$ is the orientation-preserving subgroup of $\Gamma(P)$.
\end{defn}

Kleinian Coxeter groups are considerably more restricted than general Kleinian groups. Accordingly, their arithmeticity can be decided by a straightforward criterion using combinatorial geometry; see Theorem~\ref{Thm:Vinberg}. Similarly, the invariant trace field of $\Gamma^+(P)$ can also be computed using combinatorial methods; see Corollary~\ref{Cor:PolyhedralInvariantTrace}.

The definitions and results described in the remainder of this section originate in the work of Vinberg~\cite{Vinberg67}. Our exposition follows Maclachlan and Reid~\cite[Section 10.4]{MR:Arithmetic}.

\begin{defn}\label{Def:CoxeterGraph}
Let $P \subset \HH^3$ be a Coxeter polyhedron. The \emph{Coxeter diagram} of $P$ is a finite graph whose vertices $v_1, \ldots, v_s$ are in bijection with the faces $F_1, \ldots, F_s$ of $P$. These vertices are connected by labeled edges as follows:
\begin{itemize}
    \item If $F_i, F_j$ meet at angle $\pi/2$, then there is no edge between $v_i$ and $v_j$.
    \item If $F_i, F_j$ meet at angle $\pi/n$ for $n \geq 3$, then $v_i, v_j$ are connected by an edge labeled $n$.
    \item If $F_i, F_j$ do not intersect and share an ideal vertex, then $v_i, v_j$ are connected by an edge labeled $\infty$ (corresponding to angle $0$).
    \item If $F_i, F_j$ do not intersect and do not share an ideal vertex, then $v_i, v_j$ are connected by a dotted edge labeled by a real number $\ell_{ij} = d(F_i, F_j)$. 
\end{itemize}
See Vinberg~\cite[Section 4]{Vinberg67} and Bogachev, Kolpakov, and Kontrovich~\cite[Section 2.1]{BogachevKolpakovKontorovich} for a reference on these conventions. See also Figure~\ref{Fig:DMReflectionQuotient} for an example. If $P$ has finite volume, its Coxeter diagram is connected \cite[Exercise 10.4.1]{MR:Arithmetic}.
\end{defn}

\begin{defn}\label{Def:GramMatrix}
Let $P \subset \HH^3$ be a Coxeter polyhedron with faces $F_1, \ldots, F_s$. Let $r_i$ be the reflection in face $F_i$, and let $\gamma_{ij} = r_i r_j \in \Gamma^+(P)$.
The \emph{Gram matrix} of $P$ is a symmetric $s \times s$ matrix $G(P) = [a_{ij}]$ whose diagonal entries are $a_{ii} = 2$ and whose off-diagonal entries satisfy $a_{ij} = - \trace(\gamma_{ij})$.

The off-diagonal can also be computed as follows:
\begin{itemize}
    % \item The diagonal entries are $a_{ii} = 2$.
    \item If $F_i, F_j$ meet at angle $\pi/n$, then $a_{ij} = - 2 \cos(\pi/n)$. 
    \item If $F_i, F_j$ meet at angle $0$, ie at an ideal vertex, then $a_{ij} = - 2 \cos(0) = -2$.
    \item If $F_i, F_j$ do not intersect and do not share an ideal vertex, then $a_{ij} = -2 \cosh(\ell_{ij})$.
\end{itemize}
Observe that every nonzero, non-diagonal entry of $G(P)$ corresponds to an edge of the Coxeter diagram of $P$.
\end{defn}

\begin{remark}\label{Rem:MinkowskiGram}
Gram matrices have a natural interpretation in the hyperboloid model of $\HH^3$. 
Recall that this model involves endowing $\R^4$ with an indefinite inner product $\langle \cdot, \cdot \rangle$. Given vectors $\xx = (x_1, x_2, x_3, x_4)$ and $\yy = (y_1, y_2, y_3, y_4)$, we set
\[
\langle \xx, \yy \rangle = x_1 y_1 + x_2 y_2 + x_3 y_3 - x_4 y_4,
\]
leading to the indefinite quadratic form $q(\xx) = \langle \xx, \xx \rangle$. Then $\HH^3$ can be identified with the half-hyperboloid
\[
H = \{ \xx \in \R^4 \: : \: q(\xx) = -1, \, x_4 > 0 \}.
\]
Indeed, one may compute the hyperbolic distance between $\xx, \yy \in H$ as $\cosh d(\xx, \yy) = - \langle \xx, \yy \rangle$.

Every totally geodesic plane $\Pi \subset \HH^3$ can be obtained as $\Pi = H \cap W$, where $W$ is a $3$--dimensional vector subspace. Furthermore, $W$ is the orthogonal complement of a vector $\mathbf e \in \R^4$. We call $\ee$ a \emph{normal vector} to $\Pi$.

Now, suppose $P \subset \HH^3 \cong H$ is a Coxeter polyhedron with faces $F_1, \ldots, F_s$. For each $F_i$, let $\Pi_i = W_i \cap H$ be the supporting hyperbolic plane. Let $\ee_i$ be a normal vector to $\Pi_i$, chosen so that it is outward-pointing with respect to $P$, and so that $q(\ee_i) = 1$. Now, the entries of the Gram matrix $G(P)$ can be realized as $a_{ij} = 2 \langle \ee_i, \ee_j \rangle$. Using this point of view, one can show that the Gram matrix $G(P)$ has rank $4$ and signature $(3,1)$. See \cite[Page 323]{MR:Arithmetic}.
 \end{remark}

\begin{remark}\label{Rem:Factor2Normalization}
The factor of $2$ in the formula $a_{ij} = 2 \langle \ee_i, \ee_j \rangle$, and in the bullets of Definition~\ref{Def:GramMatrix}, is a matter of convention, present in a proper subset of the literature. Compare Maclachlan and Reid~\cite[Section 10.4]{MR:Arithmetic} to  \cite{BogachevKolpakovKontorovich} and \cite{Vinberg67}.
\end{remark}

\begin{defn}\label{Def:PolyhedralFields}
Let $P \subset \HH^3$ be a Coxeter polyhedron with faces $F_1, \ldots, F_s$ and Gram matrix $G(P) = [a_{ij}]$. A \emph{multi-index} is a subset $I  \subset \{1, \ldots, s\}$, endowed with an ordering of its elements. To every multi-index $I = \{ i_1, \ldots, i_r \}$, we associate a cyclic product
\begin{equation}
  b_I = a_{i_1 i_2} a_{i_2 i_3} \cdots a_{i_r i_1}.
\end{equation}
The nonzero cyclic products correspond to closed paths in the Coxeter diagram of $P$ \cite[Exercise 10.4.1]{MR:Arithmetic}.
We associate the following number fields to $P$:
\begin{align}
K(P) &= \Q(a_{ij} : 1 \leq i,j \leq n ), \\
k(P) &= \Q(b_I : I \subset \{1, \ldots, n \} \, ).    
\end{align}
The field $k(P)$ is sometimes called the \emph{adjoint trace field} of $\orbP$. By a theorem of Vinberg~\cite{Vinberg71}, $k(P)$ is a commensurability invariant.
\end{defn}

The following elegant characterization of arithmetic reflection groups is a special case of a theorem of Vinberg~\cite[Theorem 2 and Remark 1]{Vinberg67}. Compare \cite[Theorem 10.4.5]{MR:Arithmetic}.

\begin{theorem}[Vinberg]\label{Thm:Vinberg}
Let $P \subset \HH^3$ be a noncompact, finite-volume Coxeter polyhedron, with Gram matrix $G(P) = [a_{ij}]$. Then the reflection group $\Gamma(P)$ is arithmetic if and only if the following two conditions hold: 
\begin{enumerate}[\: \: (a)]
\item\label{Itm:SimplerA} Every $a_{ij}$ is an algebraic integer.
\item\label{Itm:SimplerB} Every cyclic product $b_I$ is rational.
\end{enumerate}
\end{theorem}

\begin{proof}
Let the fields $K(P)$ and $k(P)$ be as in Definition~\ref{Def:PolyhedralFields}.
Vinberg proved the following three-part criterion for arithmeticity, which holds for both compact and non-compact polyhedra $P$:\begin{enumerate}
\item\label{Itm:Vinberg1} Every $a_{ij}$ is an algebraic integer.
\item\label{Itm:Vinberg2} $k(P)$ is totally real, meaning every embedding $k(P) \to \C$ has image in $\R$.
\item\label{Itm:Vinberg3} For every embedding $\sigma \from K(P) \to \C$ such that $\sigma \vert_{k(P)} \neq \operatorname{Id}$, the matrix $[\sigma(a_{ij})]$ is positive semi-definite.
\end{enumerate}
See Maclachlan and Reid~\cite[Theorem 10.4.5]{MR:Arithmetic} for this formulation of Vinberg's criterion.

Now, suppose that a Coxeter polyhedron $P$ satisfies \eqref{Itm:SimplerA} and \eqref{Itm:SimplerB}. Since $b_I \in \Q$ for every multi-index $I$, we have $k(P) = \Q(b_I) = \Q$, so condition \eqref{Itm:Vinberg2} holds. Furthermore, every embedding $\sigma \from K(P) \to \C$ must restrict to the identity on $k(P) = \Q$, hence condition \eqref{Itm:Vinberg3} holds vacuously. Thus Vinberg's theorem implies $\Gamma(P)$ is arithmetic.

For the converse, suppose $P$ is non-compact and $\Gamma(P)$ is arithmetic. Then $\Gamma(P)$ and $\Gamma^+(P)$ are commensurable to the Bianchi group $PSL(2, \mathcal{O}_d)$ for some $d$, hence the invariant trace field is $k \Gamma^+(P) = \Q(\sqrt{-d})$. By \cite[Lemma 10.4.2]{MR:Arithmetic}, $b_I \in k \Gamma^+(P) = \Q(\sqrt{-d})$ for every $I$. By Vinberg's condition~\eqref{Itm:Vinberg2}, $k(P)$ is totally real, hence $b_I \in \R$ for every multi-index $I$. But $\Q(\sqrt{-d}) \cap \R = \Q$, hence $b_I \in \Q$.
\end{proof}

\subsection{Gram matrices and invariant trace fields}
The Gram matrix $G(P)$ can also be used to compute the invariant trace field $k\Gamma^+(P)$.

\begin{remark}\label{Rem:PathVectors}
Let $H$ be the hyperboloid model of $\HH^3$.
Let $P \subset \HH^3 \cong H$ be a finite-volume Coxeter polyhedron with faces $F_1, \ldots, F_s$. Recall that the Coxeter diagram of $P$ is connected~\cite[Exercise 10.4.1]{MR:Arithmetic}. For every face $F_r$, let $\gamma_r$ be a path in the Coxeter diagram from the vertex corresponding to $F_1$ to the vertex corresponding to $F_r$. This path is a concatenation $\gamma_r = e_{1,i_1} e_{i_1, i_2} \cdots e_{i_{r-1}, i_r}$. Each edge $e_{ij}$ along this path corresponds to an entry $a_{ij} \in G(P)$. If $\ee_r \in \R^4$ is the unit normal vector to $F_r$, as in Remark~\ref{Rem:MinkowskiGram}, we define a vector
\[
\vv_r = a_{1,i_1} a_{i_1, i_2} \cdots a_{i_{r-1}, i_r} \ee_r.
\]
Observe that every $\vv_r$ depends on the choice of a path in the Coxeter diagram, hence  $\vv_r$ is well-defined up to scalar multiplication by an element of  $k(P)$. Recall from Definition~\ref{Def:PolyhedralFields} that $k(P)$ is generated by elements $b_I$ corresponding to loops in the Coxeter graph. 

Let $M = M(P) \subset \R^4$ be the vector space over $k(P)$ spanned by $\vv_1, \ldots, \vv_s$. This is well-defined, because the vectors $\vv_i$  are defined up to scalar multiplication in $k(P)$. Maclachlan and Reid observe~\cite[Exercise 10.4.1]{MR:Arithmetic} that $M \otimes \R = \R^4$, hence $M$ is $4$--dimensional, and the inner product $\langle \cdot, \cdot \rangle$ described in Remark~\ref{Rem:MinkowskiGram} still has signature $(3,1)$ on $M$.
Let $d \in k(P)$ be the discriminant of the quadratic form $q$, restricted to $M$. This discriminant $d$ determines the invariant trace field.
\end{remark}

The following result is \cite[Theorem 10.4.1]{MR:Arithmetic}.

\begin{theorem}\label{Thm:PolyhedralDiscriminant}
Let $P \subset \HH^3 \cong H$ be a finite-volume Coxeter polyhedron with faces $F_1, \ldots, F_s$. Let $k(P)$ be the adjoint trace field in Definition~\ref{Def:PolyhedralFields}, and let $d \in k(P)$ be the discriminant as in Remark~\ref{Rem:PathVectors}. Then the invariant trace field of $\Gamma^+(P)$ is
\[
k \Gamma^+(P) = k(P) (\sqrt{d} ).
\]
\end{theorem}

The invariant trace field $k(P) (\sqrt{d} )$ can be computed using the following practical procedure. Since $M = M(P)$ is a $4$--dimensional vector space, it is spanned by four vectors in $\{\vv_1, \ldots, \vv_s\}$. After reordering, we may assume that $\{\vv_1, \ldots, \vv_4\}$ form a basis.  Let $G'(P)$ be the $4 \times 4$ matrix whose entries are $\langle \vv_i, \vv_j \rangle$ for $1 \leq i, j \leq 4$. 
Since the rows of $G'(P)$ are well-defined up to scalar multiplication by elements in $k(P)$, and the matrix is symmetric, the determinant $\det G'(P)$ is well-defined up to multiplication by \emph{squares} of elements of $k(P)$. Based on these observations,
Maclachlan and Reid point out \cite[Section 10.4.3]{MR:Arithmetic} that $\det G'(P) = c^2 d$, where $c \in k(P)$ and $d$ is the discriminant. Since the factor of $c^2$ does not affect the field extension, it follows that
\[
k \Gamma^+(P) = k(P)(\sqrt{d}) = k(P)(\sqrt{c^2 d}).
\] 
Thus we obtain the following corollary:

\begin{corollary}\label{Cor:PolyhedralInvariantTrace}
Let $P \subset \HH^3 \cong H$ be a finite-volume Coxeter polyhedron with faces $F_1, \ldots, F_s$. Let $k(P)$ be the adjoint trace field  in Definition~\ref{Def:PolyhedralFields}. Let the vectors $\vv_1, \ldots, \vv_s$ be as in Remark~\ref{Rem:PathVectors}. Suppose, after reordering, that $\vv_1, \ldots, \vv_4$ span $M(P)$, and that $G'(P)$ is the $4 \times 4$ matrix whose entries are $\langle \vv_i, \vv_j \rangle$ for $1 \leq i, j \leq 4$. 
Then the invariant trace field of $\Gamma^+(P)$ is
\[
k \Gamma^+(P) = k(P) \left(\sqrt{ \det G'(P) } \right).
\]
\end{corollary}

A closely related computation of the invariant trace field $k\Gamma^+(P)$ via a determinant appears in Agol, Long, and Reid~\cite[Theorem 2.3]{ALR:BianchiGroups}.

\section{Quotient orbifolds and the arithmeticity of $DM$}
\label{Sec:Quotient orbifolds and arithmeticity}

In this section, we consider right-angled tiling links in $F \times I$, where $F$ has genus $g \geq 2$. 
In Corollary~\ref{Cor:SameTileCommens}, we prove if two links $L$ and $L'$ correspond to the same $[m,n,m,n]$ tiling pattern, the link exteriors cover a common orbifold and are therefore commensurable. In fact, 
the double of a link exterior $\Ext(L)$ covers a Coxeter orbifold $\orbP$ that depends only on the $[m,n,m,n]$ data of the tiling (Lemma~\ref{Lemma:PolyhedralQuotient}). 
In Lemma~\ref{Lem:Gram matrix}, we find the Gram matrices of the underlying polyhedra of the orbifold $\orbP$. Then, in Theorem~\ref{Thm:TilingArithmeticity}, we use Vinberg's criterion to determine precisely which doubled link exteriors are arithmetic. 

\subsection{Quotient orbifolds}
We begin by constructing quotient orbifolds of $\Ext(L)$ and its double.

\begin{prop}\label{Prop:OneCuspedQuotient}
Let $F$ be a surface of genus $g>1$. Let $L$ be a right-angled tiling link in $F \times I$ corresponding to an $[m, n, m, n]$ tiling $T$. Then $\Ext(L) = (F \times I) \setminus L$ covers an orientable orbifold $\orbO$ with a single cusp and with totally geodesic boundary. This orbifold can be obtained as $\orbO = ((\widetilde F \times I) \setminus \widetilde L)/G$,
    where $\widetilde L$ is the preimage of $L$ in $\widetilde F \times I$ and $G$ is the orientation-preserving and color-preserving symmetry group of the tiling $\widetilde T$ of $\widetilde F$.
\end{prop}

\begin{proof}
By Definition~\ref{Def:RightAngledLink},  $\pi(L)$ induces a regular tiling $T$ of the surface $F$ by regular $m$-gons and $n$-gons, arranged in a $[n, m, n, m]$ pattern about every vertex. The universal cover of this tiling is the unique tiling $\widetilde T$ of $\HH^2$, by regular $m$-gons and $n$-gons. 

 \begin{figure}%[ht]
    \centering
    \begin{overpic}[abs,unit=1mm,scale=.8, trim={0 4.25in 0 0}, clip]{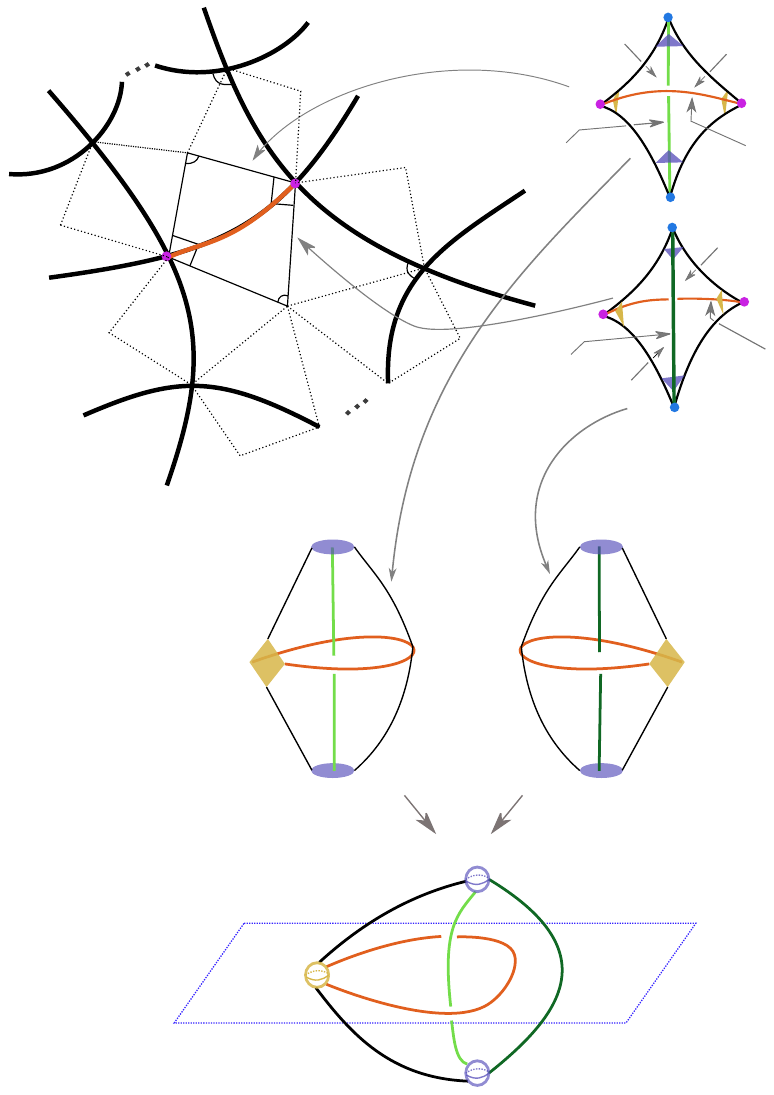}
    \put(47,137){$\alpha_n$}
    \put(70,112){$\alpha_m$}

    \put(102,145){$C$}
    \put(120,144){$C'$}
    \put(111,141){$A$}
    \put(111,132){$B$}
    \put(95,129){$\frac{2\pi}{n}$}
    \put(123,130){$\alpha_m$}

    \put(119,118){$A'$}
    \put(124,102){$\alpha_n$}
    \put(107,107){$D$}
    \put(112,106){$D'$}
    \put(95,100){$\frac{2\pi}{m}$}
    \put(104,96){$B'$}

    \put(55,70){$\alpha_n$}
    \put(55,52){$\alpha_n$}
    \put(69,67){$A$}
    \put(69,55){$B$}
    \put(77,61){$\alpha_m$}

    \put(109,70){$\alpha_m$}
    \put(109,52){$\alpha_m$}
    \put(95,67){$A'$}
    \put(95,55){$B'$}
    \put(86,61){$\alpha_n$}
    
    \put(66,26){$\pi$}
     \put(64,10){$\pi$}
     \put(70,20){$\pi$}
     \put(77,19){$\frac{2\pi}{n}$}
      \put(97,19){$\frac{2\pi}{m}$}
    \end{overpic}
    \caption{Top-left: an $[m,n,m,n]$ tiling with corresponding wedges sharing a horizontal edge shown in orange. The orientation preserving isometry group of the tiling $\widetilde{T}$ is generated by a $\frac{2\pi}{m}$ rotation about the center of the $m$-gon and a $\frac{2\pi}{n}$ rotation about the center of the $n$-gon, with a quadrilateral fundamental domain that has interior angles $\frac{2\pi}{m}$, $\frac{2\pi}{n}$, and $\frac{\pi}{2}$. Top-right: the two wedges in $\Ext(L)$ corresponding to this fundamental domain. Middle: the result of gluing each wedge to itself ($C$ to $C'$ and $D$ to $D'$). Bottom: gluing face $A$ to $A'$ and face $B$ to $B'$ produces the orbifold $\orbO$ in Proposition~\ref{Prop:OneCuspedQuotient}. The truncations of ideal vertices, which glue up to form the cusp cross-section of $\orbO$, appear in yellow. The truncation faces of the bipyramids, which glue up to form the totally geodesic boundary of $\orbO$, appear in blue.}
    \label{Fig:AngleCalcs}
\end{figure}

The symmetry group of $\widetilde T$ contains a rotation of angle $2\pi/m$ about the center of a regular $m$-gon and a rotation of angle $2\pi/n$ about the center of an adjacent $n$-gon. Together, these two rotations generate a group $G$ of isometries whose fundamental domain $Q$ is a quadrilateral with one corner in the center of the $m$-gon, one corner in the center of the $n$-gon, and two corners at vertices of $\widetilde T$. See the left panel of Figure~\ref{Fig:AngleCalcs}. 

Recall that the right-angled hyperbolic structure on $\Ext(L) = (F \times I) \setminus L$ is unique. Thus every orientation-preserving symmetry of $\widetilde T$ that preserves the colors of the two checkerboard surfaces determines an (orientation-preserving) symmetry of the pair $(\widetilde F \times I, \widetilde L)$ where $\widetilde L$ is the preimage of $L$. The color-preserving hypothesis ensures that over-crossings are sent to over-crossings, hence $\widetilde L$ is sent to $\widetilde L$.
Thus $G$ acts by orientation-preserving $3$-dimensional isometries on $(\widetilde F \times I) \setminus \widetilde L$, with quotient orbifold $\orbO = \big((\widetilde F \times I) \setminus \widetilde L \big) / G$.

We can now give a more explicit description of $\orbO$. Recall from Figure~\ref{Fig:Wedge} that every $n$-gon in the tiling $\widetilde T$ determines a regular ideal $n$--bipyramid in $(\widetilde F \times I) \setminus \widetilde L$, built as a union of wedges arranged around a stellating edge. The portion of the fundamental domain $Q$ in the $n$-gon determines a wedge in the $n$--bipyramid with two ideal vertices and two truncated vertices on $\widetilde F \times \bdy I$. Similarly, the portion of $Q$ in the $m$-gon determines a wedge in the corresponding $m$--bipyramid. See the top-right of Figure~\ref{Fig:AngleCalcs}. 

The symmetry group $G$ carries every wedge in every bipyramid to one of these two wedges. As a result, the gluing pattern for the bipyramidal decomposition of $\Ext(L)$ (which was described in \cite[Theorem 4.4]{ACM:K-uniform}) is modified in $\orbO$. In particular, since every wedge over an $n$-gon is taken to the top wedge of Figure~\ref{Fig:AngleCalcs} by the $G$-action, the top wedge acquires self-gluings along faces $C$ and $C'$. Similarly, the wedge over an $m$-gon acquires a self-gluing along faces $D$ and $D'$.

To see the construction of $\orbO$, we follow the gluing instructions of Figure~\ref{Fig:AngleCalcs}.
We first glue each wedge to itself, identifying face $C$ with $C'$ and face $D$ with $D'$ (see the middle row of Figure~\ref{Fig:AngleCalcs}). Next, we glue the two wedges to each other, identifying $A$ with $A'$ and $B$ with $B'$. The resulting orbifold $\orbO$ is shown in the bottom row of Figure~\ref{Fig:AngleCalcs}. 
\end{proof}

As an immediate consequence of Proposition~\ref{Prop:OneCuspedQuotient}, we obtain the easier direction of Theorem~\ref{Thm:MainCommensurability} for surfaces of genus $g \geq 2$. This result was previously established in \cite[Proposition 3]{RKK:RightAngled}.

\begin{corollary}
\label{Cor:SameTileCommens}
Let $L \subset (F \times I)$ and $L' \subset (F' \times I)$ be right-angled tiling links in thickened hyperbolic surfaces. If the tilings of $F$ and $F'$ lift to the same tiling of $\HH^2$, then $M = \Ext(L)$ and $M' = \Ext(L')$ are commensurable.
Consequently the doubles $DM$ and $DM'$ are also commensurable.
\end{corollary}

\begin{proof}
By Proposition~\ref{Prop:OneCuspedQuotient}, $M$ and $M'$ cover one-cusped orbifolds $\orbO$ and $\orb'$, respectively. Furthermore,
\[
\orbO = ((\widetilde F \times I) \setminus \widetilde L)/G
\]
where $G$ is the orientation-preserving color-preserving symmetry group of the tiling $\widetilde T$ of $\widetilde F = \HH^2$. Since the tiling $T'$ of $F'$ also lifts to the same tiling $\widetilde T'$ of $\HH^2$, it has the same symmetry group, hence $\orbO = \orb'$. 
Since $M$ and $M'$ both cover $\orbO = \orb'$, Remark~\ref{Rem:EasyCommensurability} implies they are commensurable.

Following Definition~\ref{Def:Silvering}, let $\orb^\pm$ be the orbifold obtained by silvering the totally geodesic boundary of $\orbO = \orb'$. Then $DM$ and $DM'$ both cover $\orb^\pm$, hence they are commensurable.
\end{proof}

Next, we construct an even simpler orbifold quotient of $DM$.

\begin{lemma}\label{Lemma:PolyhedralQuotient}
Let $F$ be a surface of genus $g>1$. Let $L$ be a right-angled tiling link in $F \times I$ corresponding to an $[m, n, m, n]$ tiling $T$. Let $DM$ be the double of 
%$M = (F \times I) \setminus L$
$M= \Ext(L)$. Then $DM$ covers a Coxeter orbifold $\orbP = \HH^3 / \Gamma(P)$, based on the polyhedron $P$ shown in Figure~\ref{Fig:DMReflectionQuotient}. Furthermore, $\bdy M$ covers the triangular face $F_6$.
\end{lemma}

\begin{proof}
Recall that $DM$ covers the silvered orbifold $\orbO^\pm$. This orbifold is depicted in the bottom of Figure~\ref{Fig:AngleCalcs} and reproduced in the top of 
Figure~\ref{Fig:DMReflectionQuotient}. As Figure~\ref{Fig:DMReflectionQuotient} shows, $\orbO^\pm$ is symmetric with respect to a horizontal plane (drawn in blue) and a vertical plane corresponding to the plane of the screen.
Reflecting in both planes produces the orbifold $\orbP$ shown in the middle two panels two of Figure~\ref{Fig:DMReflectionQuotient}. In particular, the third panel shows $\orbP$ is a Coxeter orbifold.
\end{proof}

\begin{figure}
    \begin{overpic}[abs,unit=1mm,scale=.65, trim={0in 6in 0 0in}, clip]{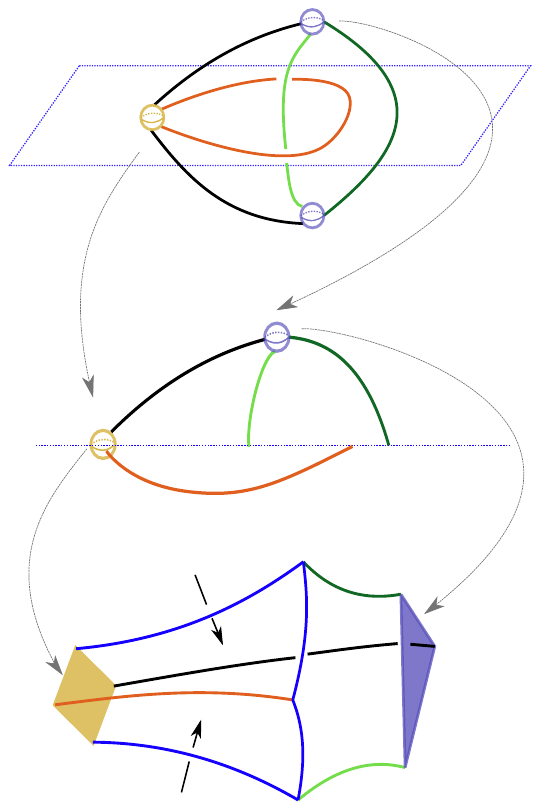}
     \put(60,88){$\pi$}
     \put(60,70){$\pi$}
     \put(62,82){$\pi$}
     \put(66,79){$\frac{2\pi}{n}$}
      \put(83,79){$\frac{2\pi}{m}$}
      
        \put(55,52){$2$}
         \put(80,51){$m$}
        \put(63,46){$n$}
         \put(60,36){$2$}
         \put(43,40){$2$}

        \put(50,24){$2$}
         \put(50,8){$2$}
         \put(54,14){$2$}
         \put(65,18){$2$}
         \put(73,23){$2$}
          \put(72,10){$2$}
          \put(77,29){$m$}
          \put(78,6){$n$}

          \put(85,13){$F_6$}
          \put(75,17){$F_1$}
          \put(55,30){$F_2$}
          \put(54,3){$F_3$}
          \put(63,11){$F_4$}
          \put(65,24){$F_5$}

     \end{overpic}

     \begin{tikzcd}
    F_4 \arrow[r,-,dashed, "\ell_{46}"] 
    \arrow[d,-,"\infty"] & F_6 \arrow[r,-,dashed, "\ell_{56}"] & F_5 \arrow[d,-,"\infty"] \\
    F_2 \arrow[r,-,"m"] & F_1 \arrow[r,-,"n"] & F_3
\end{tikzcd}

\caption{Constructing the reflection orbifold $\orbP$ as a quotient of $DM$. We begin with $\orbO^{\pm}$. We then quotient by two reflections in perpendicular planes: the dotted blue horizontal plane and the plane of the screen. The result is a Coxeter polyhedron $P$ with faces labeled $F_1$ through $F_6$. 
The final panel of the figure shows the Coxeter diagram of $P$.
\label{Fig:DMReflectionQuotient}}
\end{figure}

\begin{remark}\label{Rem:WedgeCount}
Tracing through the gluing steps in Figure~\ref{Fig:AngleCalcs} and the quotient in Figure~\ref{Fig:DMReflectionQuotient} shows that the Coxeter polyhedron $P$ consists of one quarter of an $m$--wedge and one quarter of an $n$--wedge.    
\end{remark}

We will eventually show that in a typical scenario, 
the polyhedral reflection orbifold $\orbP$ is the unique smallest orbifold in the commensurability class of $DM = D\Ext(L)$. More precisely, this holds whenever $L$ corresponds to an $[m,n,m,n]$ tiling for $m \neq n$, and furthermore $DM$ is non-arithmetic. See Corollary~\ref{Cor:MinimalOrbifold}. However, proving this will require tools that we do not yet have, specifically the canonical polyhedral decomposition of $DM$. 

\subsection{Gram matrices and arithmeticity}
Our next goals are to characterize the arithmeticity of $\orbP$ (Theorem~\ref{Thm:TilingArithmeticity})  and compute its invariant trace field (Corollary~\ref{Cor:InvariantTraceFields}). We will do both by computing the Gram matrix of $P$  
in terms of the parameters $m$ and $n$. We thank Nikolay Bogachev for his assistance with the computation.

\begin{lemma}
\label{Lem:Gram matrix}
Let $P$ be the Coxeter polyhedron of Lemma~\ref{Lemma:PolyhedralQuotient}.
Then the Gram matrix of $P$ is
\[
G(P) = \begin{bmatrix}
        2 & -2\cos(\pi/m) & -2\cos(\pi/n) &0 &0 &0\\
        -2\cos(\pi/m) & 2 & 0 & -2 & 0 & 0\\
        -2\cos(\pi/n) & 0 & 2 & 0 &-2 & 0\\
        0 & -2 & 0 & 2 & 0 & -2 C_{m,n}\\
        0 & 0 & -2 & 0 & 2 & -2C_{n,m}\\
        0 & 0 & 0 & -2C_{m,n} & -2C_{n,m} & 2\\
    \end{bmatrix}
\]
Furthermore, the terms $C_{m,n}$ and $C_{n,m}$ can be computed as follows:
\[
C_{m,n} = \frac{\cos(\pi/m)}{\sqrt{\cos(\pi/m)^2 + \cos(\pi/n)^2 - 1}},
\qquad
C_{n,m} = \frac{\cos(\pi/n)}{\sqrt{\cos(\pi/m)^2 + \cos(\pi/n)^2 - 1}},
\]
\end{lemma}

\begin{proof}
The Coxeter diagram of $P$ is shown in Figure~\ref{Fig:DMReflectionQuotient}. Since faces that meet at a $\pi/2$ angle do not lead to an edge of the diagram, the edges form a $6$--cycle as shown in the figure. Two of the edges, corresponding to the ultraparallel pairs of faces $(F_4,F_6)$ and $(F_5,F_6)$, are dashed.
%=
Following Definition~\ref{Def:GramMatrix}, we use the Coxeter diagram to build the Gram matrix $G(P)$:
\[
G(P) = \begin{bmatrix}
        2 & -2\cos(\pi/m) & -2\cos(\pi/n) &0 &0 &0\\
        -2\cos(\pi/m) & 2 & 0 & -2 & 0 & 0\\
        -2\cos(\pi/n) & 0 & 2 & 0 &-2 & 0\\
        0 & -2 & 0 & 2 & 0 & -2\cosh(\ell_{46})\\
        0 & 0 & -2 & 0 & 2 & -2\cosh(\ell_{56})\\
        0 & 0 & 0 & -2\cosh(\ell_{46}) & -2\cosh(\ell_{56}) & 2\\
    \end{bmatrix}
\]

It remains to compute the values $C_{m,n} = \cosh(\ell_{46})$ and $C_{n,m} = \cosh(\ell_{56})$ corresponding to the ultraparallel pairs of faces of $P$.
To this end, recall from \cite[Page 323]{MR:Arithmetic} that $G(P)$ has rank $4$. Thus the minor matrix $H$ obtained by removing the fifth row and column must have determinant $\det(H) = 0$. This determinant is
\[
0 = \det(H) = 2^5 \left[\cos(\pi/m)^2 + (\cosh \ell_{46})^2 \big(\cos(\pi/m)^2+ \cos(\pi/n)^2-1 \big) \right],
\]
Solving for $\cosh (\ell_{46})$ produces the value claimed in the lemma statement. 

The value of  
$C_{n,m} = \cosh(\ell_{56})$ is computed in exactly the same manner, using the determinant of the minor matrix obtained by removing the fourth row and column of $G(P)$.
\end{proof}

\begin{remark}
Due to the presence of right angles in $P$, each of the distances $\ell_{46}$ and $\ell_{56}$ is the length of an edge of $P$. Accordingly, these distances can be computed using purely geometric techniques, by a repeated application of the hyperbolic law of cosines. Having performed that calculation, we find the determinants of minors of $G(P)$ to be quicker and more straightforward.
\end{remark}

Plugging in the values $\cos(\pi/4) = \sqrt{2}/2$ and $\cos(\pi/6) = \sqrt{3}/2$ yields the following corollary of Lemma \ref{Lem:Gram matrix}.

\begin{corollary}\label{Cor:OurGram}
Let $P$ be the Coxeter polyhedron of Lemma~\ref{Lemma:PolyhedralQuotient}.
For $(m,n)=(6,4)$, the Gram matrix is 
\begin{equation} \label{Eqn:64Gram}
    G(P_{6,4}) = \begin{bmatrix}
        2 & -\sqrt{3} & -\sqrt{2} &0 &0 &0\\
        -\sqrt{3} & 2 & 0 & -2 & 0 & 0\\
        -\sqrt{2} & 0 & 2 & 0 &-2 & 0\\
        0 & -2 & 0 & 2 & 0 & -2\sqrt{3}\\
        0 & 0 & -2 & 0 & 2 & -2\sqrt{2}\\
        0 & 0 & 0 & -2\sqrt{3} & -2\sqrt{2} & 2\\
    \end{bmatrix}
\end{equation}
For $(m,n)=(6,6)$ the Gram matrix is 
\begin{equation}\label{Eqn:66Gram}
     G(P_{6,6}) = \begin{bmatrix}
        2 & -\sqrt{3} & -\sqrt{3} &0 &0 &0\\
        -\sqrt{3} & 2 & 0 & -2 & 0 & 0\\
        -\sqrt{3} & 0 & 2 & 0 &-2 & 0\\
        0 & -2 & 0 & 2 & 0 & -\sqrt{6}\\
        0 & 0 & -2 & 0 & 2 & -\sqrt{6}\\
        0 & 0 & 0 & -\sqrt{6} & -\sqrt{6} & 2\\
    \end{bmatrix}
\end{equation}
\end{corollary}

We can next calculate the invariant trace fields of the doubled link exteriors corresponding to the $[6,4,6,4]$ and $[6,6,6,6]$ tilings, following the procedure described immediately after Theorem \ref{Thm:PolyhedralDiscriminant}.

\begin{corollary}\label{Cor:InvariantTraceFields}
The invariant trace fields of the link exteriors corresponding to the $[6,4,6,4]$ and $[6,6,6,6]$ tilings are $\Q(i\sqrt{6})$ and $\Q(i)$, respectively.
\end{corollary}

\begin{proof}
We adopt the notation of Remark~\ref{Rem:PathVectors} and Corollary~\ref{Cor:PolyhedralInvariantTrace}. For each Coxeter polyhedron $P$, the vector space $M(P)$ spanned by all the vectors $\vv_i$ admits the basis $\{ \vv_1, \vv_2, \vv_3, \vv_4\}$ with $\vv_1= a_{11}\ee_1$ and $\vv_2=a_{12}\ee_2$ and $ \vv_3=a_{13}\ee_3$ and $\vv_4=a_{12}a_{24}\ee_4$. Now, we can compute the matrices $G'(P)$ whose entries are $\langle \vv_i, \vv_j \rangle$. 
    
For the $[6,4,6,4]$ tiling, we have $\vv_1=2\ee_1 $ and $ \vv_2=-\sqrt{3}\ee_2 $ and $\vv_3=-\sqrt{2}\ee_3$ and $\vv_4=(2\sqrt{3})\ee_4$, Therefore 
\[
G'(P_{6,4})=
    \begin{bmatrix}
    8 & 6 & 2\sqrt{6}&0\\
    6&6&0&12\\
    2\sqrt{6}&0&4&0\\
    0&12&0&24
    \end{bmatrix}.
\]
This matrix has  $\det(G'(P_{64})=-3456$. Thus, by Corollary~\ref{Cor:PolyhedralInvariantTrace}, the invariant trace field for $(m,n)=(6,4)$ is $k\Gamma^+(P_{6,4})(\sqrt{-3456})=\Q(24\sqrt{-6})=\Q(i\sqrt{6})$.
    
For the $[6,6,6,6]$ tiling, we have $\vv_1=2\ee_1 $ and $ \vv_2=-\sqrt{3}\ee_2 $ and $ \vv_3=-\sqrt{3}\ee_3$ and $\vv_4=(2\sqrt{3})e_4.$ Thus 
\[
G'(P_{6,6})=
    \begin{bmatrix}
    8 & 6 & 6&0\\
    6&6&0&12\\
    6&0&6&0\\
    0&12&0&24
    \end{bmatrix}.
\]
This matrix has $\det(G'(P_{6,6})=-5184.$
Therefore, the invariant trace field for for $(m,n)=(6,6)$ is $k\Gamma^+(P_{6,6})(\sqrt{-5184})=\Q(72\sqrt{-1})=\Q(i)$.
\end{proof}

We can now prove the following result, which forms the heart of Theorem~\ref{Thm:MainArithmeticity}.

\begin{theorem}\label{Thm:TilingArithmeticity}
Let $F$ be a surface of genus $g \geq 2$. Let $M = \Ext(L)$ be the exterior of a right-angled tiling link in $F \times I$, corresponding to a $[m,n,m,n]$ tiling. Then the double $DM$ is arithmetic if and only if $(m,n) = (6,4)$ or $(m, n) = (6,6)$.
\end{theorem}

\begin{proof}
By Lemma~\ref{Lemma:PolyhedralQuotient}, $DM$ covers a Coxeter orbifold $\orbP = \HH^3 / \Gamma(P)$, corresponding to the Coxeter polyhedron $P$ of Figure~\ref{Fig:DMReflectionQuotient}. Thus $DM$ is arithmetic if and only if $\orbP$ is arithmetic. Recall that the Gram matrix of $P$ was computed in Lemma~\ref{Lem:Gram matrix}. 

Now, suppose that $\orbP$ is arithmetic. By Vinberg's criterion (Theorem~\ref{Thm:Vinberg}),  every cyclic product $b_I$, as in Definition~\ref{Def:PolyhedralFields}, must be rational. In particular, arithmeticity requires
\[
b_I = a_{12} a_{21} = 4\cos^2(\pi/m) = 2 \cos(2\pi/m) +2 \in \Q
\]
and
\[
b_J = a_{13} a_{31} = 4\cos^2(\pi/n) = 2 \cos(2\pi/n) +2 \in \Q.
\]
For an integer $p \geq 3$, Niven's Theorem~\cite[Corollary 3.12]{Niven} says that $\cos(2\pi/p) \in \Q$ if and only if $p \in \{3,4,6\}$.  
Thus arithmeticity of $\orbP$ implies $\{m,n\} \subset \{3,4,6\}$.
Since the tilings $[3,3,3,3]$, $[4,3,4,3]$, $[6,3,6,3]$, and $[4,4,4,4]$ cannot occur on a hyperbolic surface, this proves the ``only if'' direction of the theorem.

For the converse direction, suppose that $(m,n)=(6,4)$. Then the Gram matrix $G(P)$ must be as shown in \eqref{Eqn:64Gram}.
It is immediate to check that every entry $a_{ij}$ is an algebraic integer, verifying condition~\eqref{Itm:SimplerA} of Theorem~\ref{Thm:Vinberg}. For condition~\eqref{Itm:SimplerB}, we need to consider the cyclic products $b_I$ corresponding to multi-indices $I \subset \{1, \ldots, 6\}$. By Figure~\ref{Fig:DMReflectionQuotient}, every nonzero cyclic product $b_I$ is either the square of a matrix entry, as in the above displayed equation, or corresponds to a cycle of all six edges in the Coxeter diagram. From \eqref{Eqn:64Gram}, we see that the square of every matrix entry is rational. Meanwhile, the cycle of all six edges gives
\begin{align*}
b_I &= a_{12} \cdot a_{24} \cdot a_{46} \cdot a_{65} \cdot a_{53} \cdot a_{31} \\
& = -\sqrt{3} \cdot -2 \cdot -2 \sqrt{3} \cdot -2 \sqrt{2} \cdot -2 \cdot - \sqrt{2} \\
& = 96 \: \in \: \Q.
\end{align*}
Thus, by Theorem~\ref{Thm:Vinberg}, $\orbP$ is arithmetic for the $[6,4,6,4]$ tiling.

Finally, suppose that $m = n = 6$. In the Gram matrix $G(P)$ shown in \eqref{Eqn:66Gram}, every entry $a_{ij}$ is an algebraic integer. The square of every matrix entry is rational. Furthermore, the cyclic product $b_I$ corresponding the cycle of all six edges is
\[
b_I \: = \:  a_{12} \cdot a_{24} \cdot a_{46} \cdot a_{65} \cdot a_{53} \cdot a_{31} \: = \: 72 \: \in \: \Q.
\]
Thus, by Theorem~\ref{Thm:Vinberg}, $\orbP$ is arithmetic for the $[6,6,6,6]$ tiling.
\end{proof}

\begin{remark}\label{Rem:RKKPreviousWork}
A handful of results in this section were previously obtained in Kaplan-Kelly's thesis~\cite{RKK:Thesis}, using different methods.
The invariant trace fields found in the proof of Corollary~\ref{Cor:InvariantTraceFields} are also computed in \cite[Theorem 5.9 and Example 5.13]{RKK:Thesis}, using the shapes of tetrahedra in ideal triangulations. We find that the use of Gram matrices in Corollary~\ref{Cor:InvariantTraceFields} simplifies the calculation. The arithmeticity of $DM$ for $(m,n)=(6,6)$ also appears in \cite[Theorem 5.9]{RKK:Thesis}, where it is shown that $DM$ has a geometric decomposition into regular ideal octahedra, and is thus commensurable to complements of the Whitehead link and the Borromean rings.
\end{remark}

\section{Canonical decompositions and commensurability}\label{Sec:Canonical}

In this section, we continue our focus on right-angled tiling links in thickened surfaces of genus $g \geq 2$. In Theorem~\ref{Thm:Canonical}, we compute the canonical decomposition of the double $D \Ext(L)$ of each link exterior. Using these canonical decompositions, we give  a complete characterization of commensurability, proving  Theorem~\ref{Thm:MainCommensurability} for surfaces of genus $g \geq 2$. In Corollary~\ref{Cor:MinimalOrbifold}, we also identify the minimal orbifold in a non-arithmetic commensurability class.

\begin{defn}\label{Def:CanonicalDecomp}
Let $M$ be a cusped hyperbolic $3$--manifold. Let $A_1, \ldots, A_n$ be a disjoint collection of horospherical neighborhoods of the cusps of $M$. The \emph{Ford--Voronoi domain} $\mathcal F \subset M$ consists of all points of $M$ that have a unique shortest path to the union of the $A_i$. The complement $\Sigma = M \setminus \mathcal F$, called the \emph{cut locus}, is a $2$--dimensional cell complex comprised of finitely many geodesic polygons. The combinatorial dual of $\Sigma$, denoted $\mathcal T$, is called the \emph{canonical polyhedral decomposition} determined by $M$ and $A_1, \ldots, A_n$. The decomposition $\mathcal T$ has one ideal polyhedron dual to each vertex of $\Sigma$, one ideal face dual to each edge of $\Sigma$, and one ideal edge dual to each polygon of $\Sigma$.
\end{defn}

Epstein and Penner~\cite{EP:Canonical} gave a characterization of $\mathcal T$ in terms of convexity in the hyperboloid model of $\HH^3$ (compare Remark~\ref{Rem:MinkowskiGram}). Accordingly, $\mathcal T$ is sometimes called the Epstein--Penner decomposition of $M$.

The term \emph{canonical polyhedral decomposition} is slightly misleading, because $\mathcal T$ is only as canonical as the choice of cusp neighborhoods in $M$. However, when $M$ has a single cusp, the choice of cusp neighborhood $A \subset M$ is immaterial because expanding or contracting $A$ does not affect the cut locus $\Sigma$ or the polyhedral decomposition $\mathcal T$. In this situation, $\mathcal T$ is completely canonical. Similarly, if $f \from M \to \orbO$ is a covering map where $\orbO$ is an orbifold with a single cusp, we may choose a horospherical cusp neighborhood $A \subset \orbO$. Then $f^{-1}(A)$ is an equivariant collection of cusp neighborhoods in $M$, with the property that expanding or contracting $A$ has no effect on the cut locus $\Sigma \subset M$ or the dual decomposition $\mathcal T$.

\begin{defn}\label{Def:TotallyCanonicalDecomp}
Let $M = \HH^3 / \Gamma$ be a cusped, non-arithmetic hyperbolic $3$--manifold with $n$ cusps. By Theorem~\ref{Thm:MargulisArithmeticity}, the commensurator $\Comm(\Gamma)$ is discrete, and the quotient $\omin = \HH^3 / \Comm(\Gamma)$ is the unique minimal orbifold covered by $M$. Choose equal-volume cusp neighborhoods in $\omin$. Pulling back these cusp neighborhoods to $M$ gives 
a \emph{maximally symmetric} collection of cusps  $A_1, \ldots, A_n \subset M$. By construction, the collection $A_1, \ldots, A_n$ is uniquely determined up to equivariant expansion or contraction, hence the cut locus $\Sigma \subset M$ is also uniquely determined. The \emph{maximally symmetric canonical decomposition} $\mathcal T$ of $M$ is the canonical polyhedral decomposition dual to $\Sigma$, determined by the maximally symmetric collection of cusps in $M$.
\end{defn}

Canonical decompositions give the following useful criterion for commensurability.

\begin{theorem}\label{Thm:GHH+Margulis}
Let $M$ and $M'$ be cusped, non-arithmetic hyperbolic $3$-manifolds. Then the following are equivalent:
\begin{enumerate}[\:\: $(1)$]
\item\label{Itm:Commens} $M$ and $M'$ are commensurable.
\item\label{Itm:CommonQuot} $M$ and $M'$ cover a common $3$--orbifold.
\item\label{Itm:GHH} The maximally symmetric canonical decompositions $\mathcal T$ of $M$ and $\mathcal T'$ of $M'$ lift to isometric tilings of $\HH^3$.
\end{enumerate}
\end{theorem}

\begin{proof}
The equivalence $\eqref{Itm:Commens} \Leftrightarrow \eqref{Itm:CommonQuot}$ is Theorem~\ref{Thm:MargulisArithmeticity}.\eqref{Itm:MargulisNonAr}, due to Margulis.
The equivalence $\eqref{Itm:CommonQuot} \Leftrightarrow \eqref{Itm:GHH}$ is a theorem of Goodman, Heard, and Hodgson \cite[Theorems 2.4 and 2.6]{GHH:Commensurators}. See also the discussion immediately after \cite[Theorem 2.6]{GHH:Commensurators}.
\end{proof}

We can now show that the canonical decompositions of right-angled tiling link exteriors consist of ideal drums. The following result is a more precise version of Theorem~\ref{Thm:MainCanonical}.

\begin{theorem}    \label{Thm:Canonical}
Let $F$ be a surface of genus $g \geq 2$, and let $L \subset F \times I$ be a right-angled tiling link corresponding to a $[m,n,m,n]$ tiling of $F$. Then the double $DM$ of $M = \Ext(L)$ covers a one-cusped orbifold $\orbP$. 
There is a choice of cusp neighborhoods of $DM$, equivariant with respect to the cover $DM \to \orbP$, such that the canonical polyhedral decomposition of $DM$ consists of regular ideal $m$--drums and $n$--drums.

Furthermore, if $DM$ is non-arithmetic, then the canonical decomposition consisting of these drums is maximally symmetric.
\end{theorem}

\begin{proof}
By Lemma~\ref{Lemma:PolyhedralQuotient}, there is a finite cover $f \from DM \to \orbP$, where $\orbP$ is a polyhedral orbifold with a single cusp.
Let $C \subset \orbP$ be an embedded horospherical neighborhood of this single cusp.
The preimage $f^{-1}(C)  \subset DM$ is an equivariant choice of cusp neighborhoods in $DM$. We use this choice of cusps in $DM$ to build the cut locus $\Sigma$ and the canonical decomposition $\mathcal T$. Observe that if $DM$ is non-arithmetic, the canonical decomposition $\mathcal T$ will necessarily be maximally symmetric, since the connected cusp $C \subset \orbP$ must cover a connected cusp in the minimal orbifold $\omin$.

By Lemma~\ref{Lem:Drums}, $DM$ admits a geometric decomposition into regular ideal $m$--drums and $n$--drums. Furthermore, each $n$--drum $\Delta$ is assembled from $2n$ half-wedges, and the quotient of $\Delta$ by its symmetry group is a quarter-wedge. By Remark~\ref{Rem:WedgeCount}, a fundamental domain for $\orbP$ consists of a quarter of an $m$--wedge and a quarter of an $n$--wedge.

We use the universal covering $p \from \HH^3 \to DM$ and the finite cover $f \from DM \to \orbP$ to pull back several objects to $\HH^3$. The cusp neighborhood $C \subset \orbP$ pulls back to a collection of horoballs in $\HH^3$. The cut locus $\Sigma \subset DM$ and the canonical decomposition $\mathcal T$ dual to $\Sigma$ pull back to 
a $2$--complex $\widetilde \Sigma \subset \HH^3$ and a polyhedral decomposition $\widetilde {\mathcal T}$ that is dual to $\widetilde \Sigma$. The decomposition of $DM$ into regular ideal drums pulls back to a tiling of $\HH^3$ by regular ideal drums.

    \begin{figure}[h]
    \centering
    \begin{overpic}[abs,unit=1mm,scale=.5, trim={ 0in 5in 0in 0in}, clip]{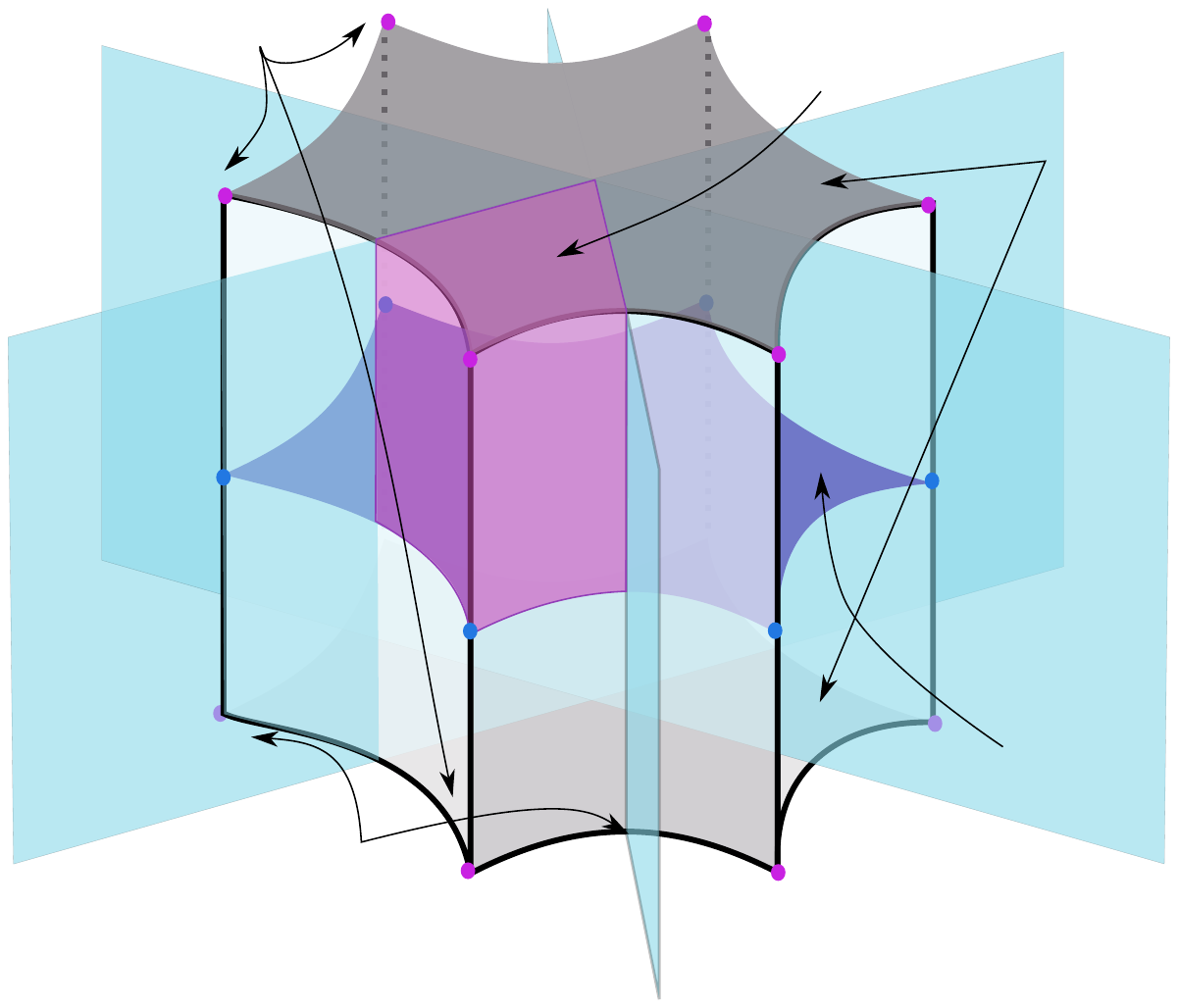}
    \put(16,81){Ideal}
    \put(72,78){$Q_i$}
    \put(80, 72){Horizontal midsections}
    \put(82,15){Truncation face}
    \put(13,5){Vertical planes of symmetry}
        \end{overpic}
        \caption{A drum with vertical and horizontal reflection planes drawn in blue. A basis of attraction is highlighted in pink.}
        \label{Drum with planes}
    \end{figure}

Given a regular ideal $n$--drum $\Delta$ in the tiling of $\HH^3$, let $H_1, \ldots, H_{2n}$ be the horoballs about the vertices of $\Delta$. Since $\Delta$ covers one of the quarter-wedges that make up the polyhedral orbifold $\orbP$, the symmetry group $\isom(\Delta)$ acts transitively on these horoballs. We associate to each horoball $H_i$ a \emph{basin of attraction} $Q_i$, namely the set of all points $y \in \Delta$ that are strictly closer to $H_i$ than to any $H_j$ for $j \neq i$. Since $\Delta$ is regular, meaning maximally symmetric, the region $Q_i$ is cut out by planes of symmetry of $\Delta$. Furthermore, $\Delta \setminus (\bigcup Q_i)$ is a union of totally geodesic polygons, with each polygon lying in a plane of symmetry,  equidistant to two of the horoballs. 
See Figure \ref{Drum with planes}, where one basin $Q_i$ is highlighted in pink. 

To prove the theorem, we must show that $\widetilde {\mathcal T}$ coincides with the tiling of $\HH^3$ by regular ideal drums. By duality, this is equivalent to showing that the lift $\widetilde \Sigma$ of the cut locus $\Sigma \subset \bdy M$ coincides with the union of the polygons of $\Delta \setminus (\bigcup Q_i)$, as $\Delta$ ranges over all the drums in the tiling. Another way to say this is that the every component of $\HH^3 \setminus \widetilde \Sigma$ is the union of all basins $Q_i$ associated to a single horoball $H$.

Consider an arbitrary point $x \in \HH^3 \setminus \widetilde \Sigma$. By Definition~\ref{Def:CanonicalDecomp}, this means there is a unique horoball $H_x \subset (f \circ p)^{-1}(C)$ that is closest to $x$. Let $\gamma = \gamma_x$ be the distance-realizing geodesic from $x$ to $H_x$. If $x \in H_x$, then $\gamma = \{x\}$ is a single point. Otherwise, some initial portion of the geodesic $\gamma$ is contained in some basin $Q$, whose ideal point is inside horoball $H_x$. 

We claim that the entire geodesic $\gamma = \gamma_x$ is contained in the basin $Q$. This will imply $x \in Q$, and since $x$ is arbitrary this implies the entire component of $\HH^3 \setminus \widetilde \Sigma$ centered around horoball $H_x$ is the union of basins that meet $H_x$.

 \begin{figure}[h]
    \centering
    \begin{overpic}[abs,unit=1mm,scale=.5, trim={0 6.5in 0 0}, clip]{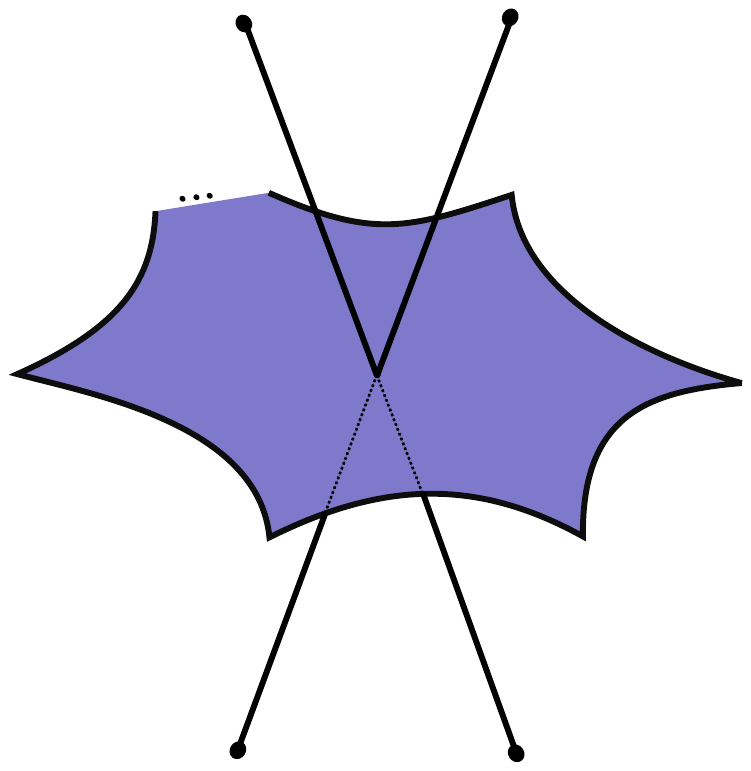}
    \put(40, 50){$\gamma$}
    \put(60, 50){$\bar{\gamma}$}
    \put(37, 63){$x$}
    \put(36, 0){$\bar{x}$}
    \put(64,0){$H_x$}
    \put(64, 63){$H_{\bar{x}}$}
    \put(51, 40){$\delta$}
        \end{overpic}
        \caption{A face in reflection plane $\Pi$ and its intersection with geodesics $\gamma$ and $\bar{\gamma}$.}
        \label{Fig: Face with intersection}
    \end{figure}

Suppose for a contradiction that $\gamma$ exits the basin $Q$ that contains its initial segment. Then $\gamma$ must have a transverse intersection with $\Pi$, a plane of reflective symmetry of the ambient drum. See Figure~\ref{Fig: Face with intersection}.
The reflection of $\gamma$ in $\Pi$ is a geodesic $\overline{\gamma}$ that connects a point $\bar{x}$ and a horoball $H_{\bar{x}}$ that lie on opposite sides of $\Pi$.
Observe that $\gamma$ and $\bar{\gamma}$ intersect at a point in $\Pi$. (If $\gamma$ is contained in an edge of one of the drums, meeting $\Pi$ at a right angle, then $\bar{\gamma}$ is parameterized in the opposite direction and thus they still meet at a point.) Consider the piecewise geodesic $\delta$, running between $x$ and $H_{\bar{x}}$, built from the initial portion of $\gamma$ and the final portion of $\overline{\gamma}$. Now, the length of $\delta$ is equal to the length of $\gamma$, which contradicts the assumption that $H_x$ is the closest horoball to $x$ because $\delta$ is only piecewise geodesic.
This proves the claim and the theorem.
\end{proof}

As a corollary of Theorem~\ref{Thm:Canonical}, we can identify the minimal orbifold in a non-arithmetic commensurability class.

\begin{corollary}\label{Cor:MinimalOrbifold}
Let $F$ be a surface of genus $g \geq 2$. Let $L \subset F \times I$ be a right-angled tiling link corresponding to a $[m,n,m,n]$ tiling, and suppose $M = \Ext(L)$ is non-arithmetic. Let $\orbP$ be the Coxeter orbifold covered by $D M$, as in Figure~\ref{Fig:DMReflectionQuotient}, and let $\omin$ be the minimal orbifold in the commensurability class of $DM$. Then the degree of the cover $\orbP \to \omin$ is
\[
d = \begin{cases}
1 & \text{if } m \neq n, \\
2 & \text{if } m = n.
\end{cases}
\]
In particular, if $m \neq n$, then $\omin = \orbP$.
\end{corollary}

\begin{proof}
Let $\Gamma = \Gamma(P) = \pi_1(\orbP)$. 
By Theorem~\ref{Thm:MargulisArithmeticity}, the minimal orbifold in the commensurability class of $DM$ and $\orbP$ is the commensurator quotient
$\omin = \HH^3 / \Comm(\Gamma)$. 

Let $\mathcal T$ be the maximally symmetric canonical decomposition of $DM$.  By a theorem of Goodman, Heard, and Hodgson~\cite[Theorem 2.6]{GHH:Commensurators}, the commensurator $\Comm(\Gamma)$ can be identified as the full symmetry group of $\widetilde {\mathcal T}$, the lift of $\mathcal T$ to the universal cover $\HH^3$.

By Theorem~\ref{Thm:Canonical}, $\mathcal T$ and $\widetilde {\mathcal T}$ consist of regular ideal $m$--drums and $n$--drums. Recall from Figure~\ref{Fig:DMReflectionQuotient} that a fundamental domain for $\orbP$ consists of one quarter of an $m$--wedge (a quotient of an $m$--drum) and one quarter of an $n$--wedge (a quotient of an $n$--drum). In particular $\Gamma = \pi_1(\orbP)$ acts transitively on the set of $m$--drums and the set of $n$--drums.

We may label every drum of $\widetilde{\mathcal T}$ as white or shaded, matching the color of the region of $\pi(L)$ corresponding to the drum. Suppose for concreteness that the $m$--drums are white and the $n$--drums are shaded.

We claim that the color-preserving subgroup of $\Comm(\Gamma)$ coincides with $\Gamma$. To that end, consider an arbitrary color-preserving element $\gamma \in \Comm(\Gamma)$. After replacing $\gamma$ by another element of the same $\Gamma$--coset, we may assume that there is an $n$--drum $\Delta$ such that $\gamma \Delta = \Delta$. Since the bases of each $n$--drum are glued to (white) $n$--drums while the lateral faces are glued to (shaded) $m$--drums, we know that $\gamma$ sends the  bases of $\Delta$ to bases. Thus $\gamma$ acts on $\Delta$ as one of the $4n$ base-preserving symmetries described in Definition~\ref{Def:Drum}. But each of these $4n$ symmetries already belongs to $\Gamma = \pi_1(\orbP)$, since $\Delta$ covers an $n$--wedge with degree $4n$. Thus $\gamma \in \Gamma$, as claimed.

If $m \neq n$, then every element of $\Comm(\Gamma)$ must be color-preserving, hence $[\Comm(\Gamma):\Gamma] = 1$. 

For the rest of the proof, we assume $m = n$. It follows that $m = n > 4$, because a $[4,4,4,4]$ tiling is Euclidean whereas $F$ is a hyperbolic surface. Consequently, the base of an $m$--drum or $n$--drum has more than $4$ sides, hence any element $\gamma \in \Comm(\Gamma)$ must send the bases of drums to bases and lateral faces to lateral faces. Since the bases of each drum are glued to other drums of the same color, every element $\gamma \in \Comm(\Gamma)$ either preserves the color of every drum, or switches the color of every drum. Thus  $[\Comm(\Gamma):\Gamma] \leq 2$. To see that $[\Comm(\Gamma):\Gamma] = 2$, it suffices to exhibit a color-reversing element of $\Comm(\Gamma)$. Given $m=n$, one of these elements acts by reflection on the polyhedron $P$ of Figure~\ref{Fig:DMReflectionQuotient}, sending face $F_2$ to $F_3$ and face $F_4$ to $F_5$.
\end{proof}

We can now prove the following result, which is the main case of Theorem~\ref{Thm:MainCommensurability}.

\begin{theorem}\label{Thm:HighGenusCommensurability}
Let $L \subset (F \times I)$ and $L' \subset (F' \times I)$ be right-angled tiling links in thickened surfaces of genus $g \geq 2$. Then the doubles $DM = D\Ext(L)$ and $DM' = D\Ext(L')$ are commensurable if and only if the tilings of $F$ and $F'$ 
have the same $m$-gon and $n$-gon faces. 
\end{theorem}

\begin{proof}
If $F$ and $F'$ are tiled by the same $m$-gon and $n$-gon faces, those tilings lift to isometric tilings of $\HH^2$. By Corollary~\ref{Cor:SameTileCommens}, $DM$ and $DM'$ cover a common $3$--orbifold, and are commensurable.

For the converse, suppose $DM$ and $DM'$ are commensurable. By Definition~\ref{Def:Arithmetic}, arithmeticity is a commensurability invariant. Thus either $DM$ and $DM'$ are both arithmetic, or neither is arithmetic.

First, suppose $DM$ and $DM'$ are arithmetic. In this case, Theorem~\ref{Thm:TilingArithmeticity} says that $F$ is tiled by $m$-gons and $n$-gons, where $(m,n) = (4,6)$ or $(m,n) = (6,6)$, and similarly for $F'$. By Corollary~\ref{Cor:InvariantTraceFields}, the reflection orbifolds for those tilings have distinct invariant trace fields. Since the commensurable manifolds $DM$ and $DM'$ must have the same invariant trace field, it follows that both $L$ and $L'$ must come from the same $(m,n)$ tiling.

Now, suppose that $DM$ and $DM'$ are non-arithmetic. Let $\mathcal T$ and $\mathcal T'$ be the maximally symmetric canonical decompositions of $DM$ and $DM'$, respectively. By Theorem~\ref{Thm:GHH+Margulis}, these canonical decompositions lift to isometric decompositions of $\HH^3$. If $M = (F \times I) \setminus L$ comes from a $[m,n,m,n]$ tiling of the surface $F$, then Theorem~\ref{Thm:Canonical} says that $\mathcal T$ consists entirely of $m$-drums and $n$-drums. Thus $\mathcal T'$ also consists of $m$-drums and $n$-drums, hence $M'$ must also come from a $[m,n,m,n]$ tiling.
\end{proof}

 \section{Tiling links on surfaces of every genus}
 \label{Sec:Lower-genus}
In this section, we extend the methods of the previous sections to determine the arithmeticity, commensurability classes, and canonical decompositions of tiling links corresponding to tilings of $T^2$ and $S^2$.  By combining those results with what we have shown about links on hyperbolic surfaces, we obtain proofs of Theorems~\ref{Thm:MainArithmeticity} and \ref{Thm:MainCommensurability}.

 \subsection{Links projecting to $T^2$}

Recall from Definitions~\ref{Def:TilingLink} and~\ref{Def:RightAngledLink} that a tiling link on a surface $F$ corresponds to a $4$--valent tiling of $F$ by regular polygons, whereas a right-angled tiling link requires a $[m,n,m,n]$ pattern at every vertex. In this subsection, we consider both right-angled and non-right-angled links on $T^2$.

If $L$ is a right-angled tiling link in $T^2 \times I$, 
the only possibility for the tiling is $(m,n) = (4,4)$ or $(m,n) = (6,3)$. The arithmeticity picture of these right-angled links is already known. Indeed, by a theorem of Champanerkar, Kofman, and Purcell~\cite[Theorem 4.1]{CKP:Biperiodic}, the exteriors of $[4,4,4,4]$ and $[6,3,6,3]$ tiling links are arithmetic. 

Turning to canonicity, we can prove Theorem~\ref{Thm:CanonicalOnTorus}, which is an analogue of Theorem~\ref{Thm:Canonical} for the class of all tiling links on $T^2$ (not necessarily right-angled). 
The proof of this result requires the following definition (compare \cite[Definition 3.6]{FP:Links}).
Let $t \subset \HH^2$ be an ideal triangle. The \emph{midpoint} of an edge $e \subset t$ is the unique point $p \in e$ such that a geodesic from $p$ to the opposite vertex of the triangle meets $e$ perpendicularly.

\begin{named}{Theorem~\ref{Thm:CanonicalOnTorus}}
Let $L \subset T^2 \times I$ be a tiling link.
Then there is a natural choice of cusp neighborhoods in $M = \Ext(L)$ such that the canonical polyhedral decomposition of $M$ consists of regular ideal tetrahedra and regular ideal octahedra. 
\end{named}
 
\begin{proof}
Recall that $M$ decomposes into geometric bipyramids, whose midsections correspond to the polygons of the tiling. (See
\cite[Theorem 4.4]{ACM:K-uniform} and \cite[Theorem 3.5]{CKP:Biperiodic}.)
We pull back this decomposition of $M$ to obtain a decomposition of $\HH^3$ into regular ideal bipyramids.

By assumption, the link $L$ corresponds to a regular Euclidean tiling, so the collection of ideal bipyramids in this decomposition are centered on triangles, squares, or hexagons \cite[Lemma 3.3]{CKP:Biperiodic}. For square horizontal midsections, the ideal bipyramids are regular ideal octahedra. For triangular horizontal midsections, each ideal bipyramid consists of two regular ideal tetrahedra. For hexagonal horizontal midsections, each ideal bipyramid consists of six regular ideal tetrahedra glued along a central stellating edge. In all of these cases, the boundary of each bipyramid consists of ideal triangles that are glued midpoint to midpoint (see Figure~\ref{Fig:Horoball expansion}). Thus, for any edge $e$ of the bipyramid decomposition of $M$, there is a well-defined midpoint, independent of the choice of triangle adjacent to $e$.

We can use the midpoints of edges to determine a choice of horoballs in $\HH^3$ that will project to a disjoint collection of cusps in $M$.
In $\HH^3$, we expand a horoball about each ideal vertex $v$ of each polyhedron, until that horoball meets the midpoints of the edges into $v$. As shown in Figure~\ref{Fig:Horoball expansion}, this happens simultaneously for all edges into $v$. Furthermore, since the bipyramids are glued to one another by isometries of ideal triangles, and the gluing sends midpoints to midpoints, the choice of horoball expansion is consistent as we move from one polyhedron to the next. The deck group $\Gamma = \pi_1(M)$ acts by isometry on the polyhedra, so the choice of horoballs is naturally $\Gamma$--equivariant, and defines a choice of cusps in $M$. We construct the canonical decomposition of $M$ with respect to this choice of cusps. 

\begin{figure}
    \centering
        \begin{overpic}[width=3.5in, trim={ 0in 7.25in 0in 0in}, clip]{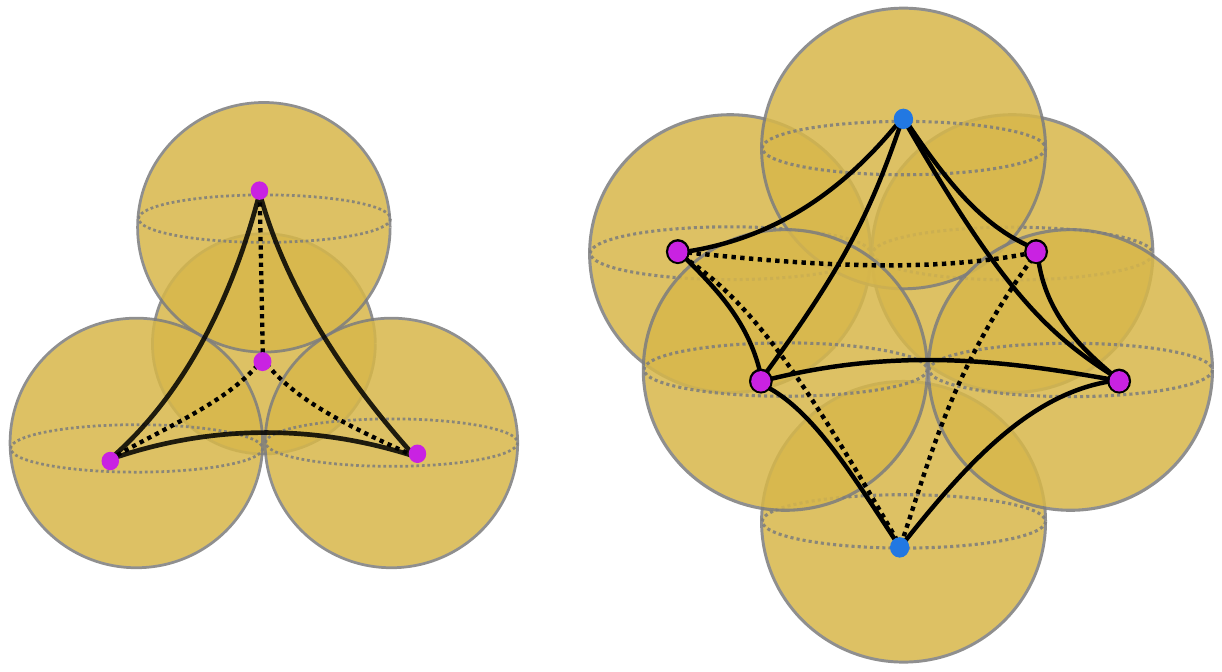}
        \end{overpic}
        \caption{An ideal regular tetrahedron and ideal regular octahedron with horoballs at each vertex. The horoballs are pairwise tangent at the midpoints of edges of the ideal polyhedra.}
        \label{Fig:Horoball expansion}
    \end{figure}

As in the proof of Theorem~\ref{Thm:Canonical}, we subdivide each regular ideal tetrahedron or octahedron $\Delta$ into basins of attraction, one for each ideal vertex. The boundary of each basin $Q$ consists of totally geodesic polygons, with each polygon contained in a plane of reflective symmetry of the ambient regular ideal tetrahedron or octahedron. See Figure~\ref{Fig:Bipyr with planes}. 

     \begin{figure}[h]
    \centering
        \begin{overpic}[width=5.5in, trim={ 0in 7.3in 0in 0in}, clip]{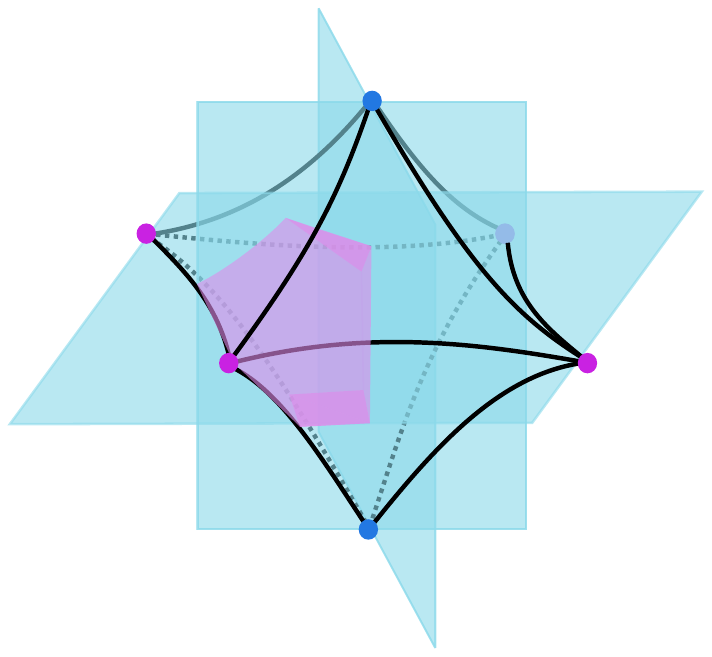}
        \put(190,105){$Q$}
        \put(150,80){$v$}
        \end{overpic}
        \caption{A regular ideal octahedron with the basin of attraction $Q$, corresponding to ideal vertex $v$, shown in pink. Reflection planes are drawn in blue.}
        \label{Fig:Bipyr with planes}
    \end{figure}

The cut locus $\Sigma \subset M$ and the canonical decomposition $\mathcal T$ dual to $\Sigma$ pull back to a $2$--complex $\widetilde \Sigma \subset \HH^3$ and a polyhedral decomposition $\widetilde {\mathcal T}$ that is dual to $\widetilde \Sigma$. We claim that $\widetilde {\mathcal T}$ coincides with our given tiling of $\HH^3$ by regular ideal octahedra and tetrahedra. As in the proof of Theorem~\ref{Thm:Canonical}, we show this by proving that each component of $\HH^3 \setminus \widetilde \Sigma$ is the union of all of the basins of attraction of the horoballs.

Take an arbitrary point $x \in \HH^3 \setminus \widetilde \Sigma$ and again let $H_x$ be the closest horoball to $x$. 
We claim that a distance-realizing geodesic $\gamma$ from $x$ to $H_x$ cannot leave the basin $Q$ that contains its initial segment. If $\gamma$ leaves $Q$, then it does so by crossing a face of $Q$ transversely. As mentioned above, each face of $Q$ is contained in a plane of reflection $\Pi$ of the ambient ideal tetrahedron or octahedron (Figure~\ref{Fig:Bipyr with planes}). Let $\overline \gamma$ be the reflection of $\gamma$ in $\Pi$. Then, exactly as in Figure~\ref{Fig: Face with intersection}, we may use a portion of $\gamma$ and a portion of $\overline \gamma$ to build a shorter path from $x$ to a horoball, obtaining a contradiction. This proves the claim and the theorem.
\end{proof}

\begin{remark}
Theorem \ref{Thm:CanonicalOnTorus} describes a canonical decomposition for the exteriors of tiling links in $T^2 \times I$. However, this canonical decomposition is not necessarily maximally symmetric. This is because the quotient of $M = \Ext(L) = T^2 \times (0,1) \setminus L$ by its symmetry group typically contains at least two cusps: one corresponding to components of $L$, and the other corresponding to the tori $T^2 \times \{0,1\}$.
Unlike the higher genus setting, we cannot identify a one-cusped quotient.
\end{remark}

\begin{remark}
The tiling links in $T^2 \times I$ that admit decompositions with at least one regular ideal tetrahedron and one regular ideal octahedron are examples of the \textit{mixed-platonic links} studied by Chesebro, Chu, DeBlois, Hoffman, Mondal, and Walsh \cite{CCDHMW:Mixed-platonic}. For mixed-platonic links, these authors independently prove that the decomposition into regular ideal tetrahedra and octahedra is canonical
\cite[Corollary 6.6]{CCDHMW:Mixed-platonic}.
\end{remark}

\subsection{Links projecting to $S^2$}
When the projection surface is $F = S^2$, there are are three tilings with a $[m,n,m,n]$ pattern: $[3,3,3,3]$, $[4,3,4,3]$, and $[5,3,5,3]$. These tilings define link exteriors in $S^3$. Hatcher showed that the first two link exteriors are arithmetic~\cite{Hatcher:Arithmetic}. Thus, to prove Theorem~\ref{Thm:MainArithmeticity}, it remains to sort out the arithmeticity of $[5,3,5,3]$. Along the way, we identify the canonical decomposition of this link exterior.

\begin{prop}\label{Prop:S^3Arithmeticity}
Let $L \subset S^3$ be the right-angled tiling link corresponding to the $[5,3,5,3]$ tiling of $S^2$. Then $\Ext(L) = S^3 \setminus L$ is non-arithmetic. Furthermore, the maximally symmetric canonical decomposition of $\Ext(L)$ consists of two ideal icosidodecahedra that meet along the checkerboard surfaces of the projection diagram $\pi(L)$.
\end{prop}

\begin{proof}
We begin by showing that $M = S^3\setminus L$ covers the Coxeter orbifold $\mathcal{P}=\HH^3 / \Gamma(P)$ corresponding to the Coxeter polyhedron $P$ in Figure \ref{Fig:53-polyhedron}. To show this, we follow exactly the same gluing and quotienting process shown in Figures~\ref{Fig:AngleCalcs} and~\ref{Fig:DMReflectionQuotient}, except with the ultra-ideal apexes replaced by finite apexes. We substitute $m=5$ and $n=3$. Figure~\ref{Fig:53-polyhedron} shows the wedges at the beginning of the construction and the polyhedron $P$ at the end; all of the intermediate steps are exactly the same.

\begin{figure}
    \centering
    \begin{overpic}[abs,abs,unit=1mm,scale=.8,trim={.5in 8in 0in 0in}, clip]{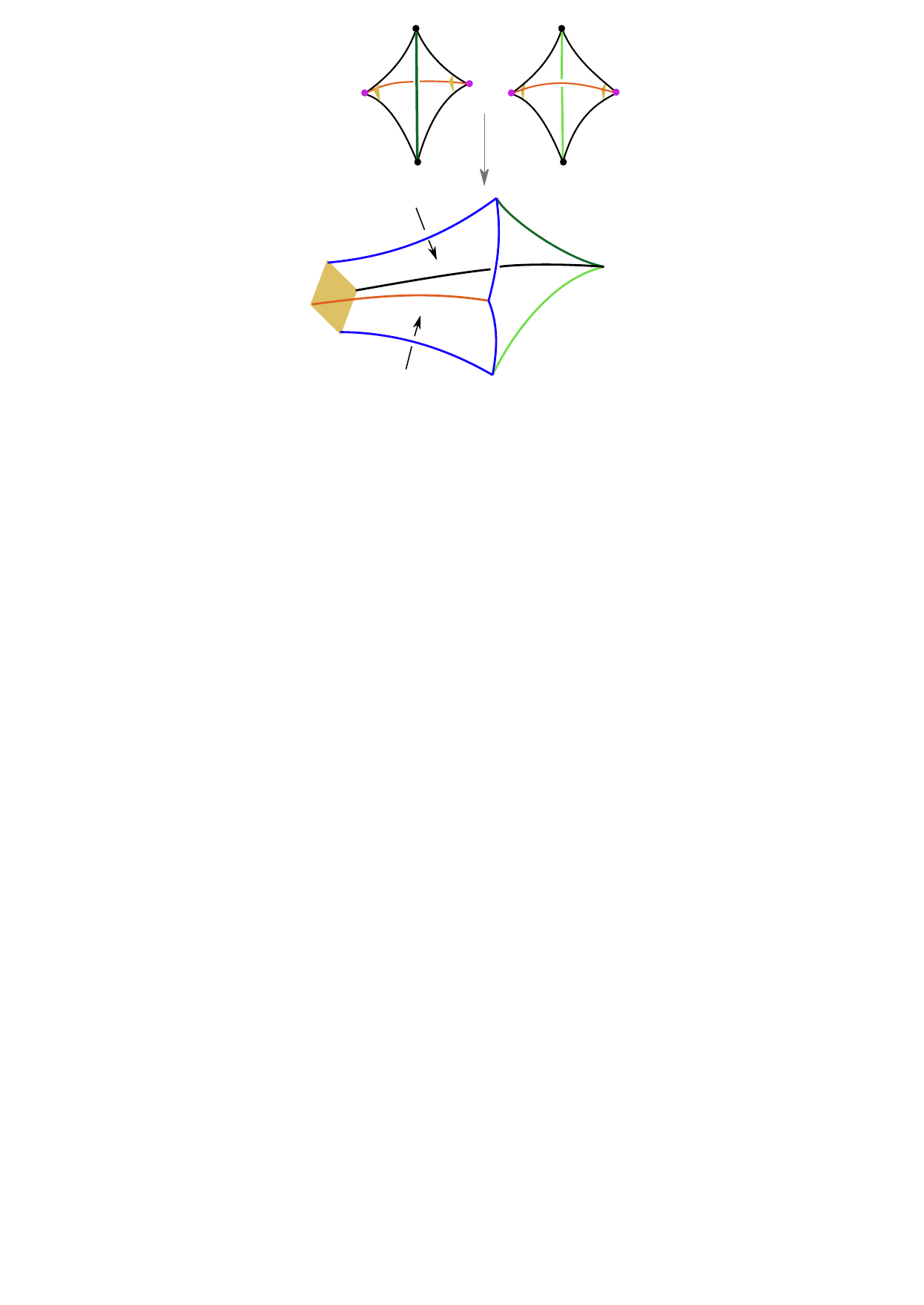}
    \put(62,39){$F_2$}
    \put(60,5){$F_3$}
    \put(72,28){$F_5$}
    \put(70,15){$F_4$}
    \put(83,21){$F_1$}
    \put(62,56){$\frac{2\pi}{5}$}
    \put(93,56){$\frac{2\pi}{3}$}
    \put(58,65){$\frac{\alpha_5}{2}$}
    \put(85,65){$\frac{\alpha_3}{2}$}
    \put(51,29){$2$}
    \put(52,11){$2$}
    \put(58,18){$2$}
    \put(72,22){$2$}
    \put(81,29){$2$}
    \put(81,15){$2$}
    \put(89,33){$5$}
    \put(89,17){$3$}
    \end{overpic}

    \begin{tikzcd}
    F_4 
    \arrow[d,-,"\infty"] &  & F_5 \arrow[d,-,"\infty"] \\
    F_2 \arrow[r,-,"5"] & F_1 \arrow[r,-,"3"] & F_3
\end{tikzcd}

    \caption{The top panel shows wedges corresponding to the $[5,3,5,3]$ tiling. The black vertices of the wedges (at the apexes) are finite while the pink vertices (on the horizontal edge) are ideal. The middle panel is the Coxeter polyhedron $P$ for the $[5,3,5,3]$ tiling link, with faces labeled $F_1$ through $F_5$. The bottom panel is the Coxeter diagram of $P$.}
    \label{Fig:53-polyhedron}
\end{figure}
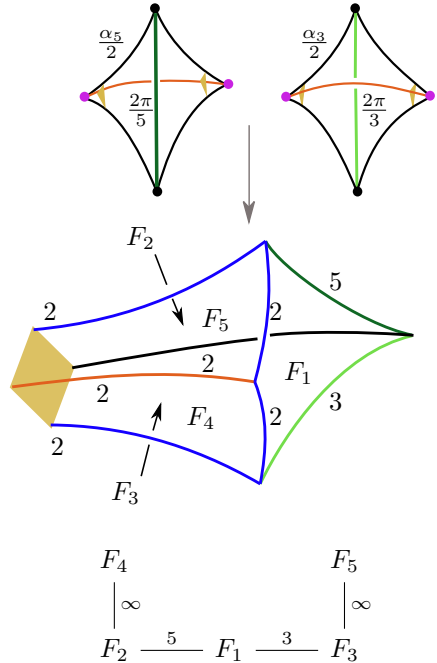

To see that $M$ is non-arithmetic, we will show that $\mathcal{P}$ is non-arithmetic. Using the Coxeter diagram of $P$ (Figure~\ref{Fig:53-polyhedron}), we build the Gram matrix as in Definition \ref{Def:GramMatrix}. Since faces $F_1$ and $F_2$ meet at a $\pi/5$ angle, we know that the Gram matrix of $P$ has entries $a_{12}=a_{21} = -2\cos(\pi/5)$. Therefore it has cyclic product $b_I=a_{12}a_{21}=(-2\cos(\pi/5))^2=2(\cos(2\pi/5)+1)$, which is  irrational because $\cos(2\pi/5)$ is irrational. Thus, by Vinberg's criterion (Theorem \ref{Thm:Vinberg}), $\orbP$ is non-arithmetic.

To find the maximally symmetric canonical polyhedral decomposition of $M$, we follow the same argument as in Theorem~\ref{Thm:Canonical}.
Let $C \subset \orbP$ be an embedded horospherical neighborhood of the single cusp of $P$. The preimage of $C$ in $M$ is an equivariant choice of cusp neighborhoods in $M$. We build the canonical decomposition with respect to this choice of cusps.

Next, decompose $M$ into two ideal polyhedra following Menasco and Thurston's polyhedral decomposition for alternating links. The result is two ideal icosidodecahedra. By a theorem of Gan~\cite[Theorem 3.14]{Gan:Alternating}, this is a geometric decomposition of the complement and the face normals in each polyhedron intersect at a common point. Both ideal polyhedra admit planes of reflective symmetry which intersect in the point at which all of the face normals of the polyhedron intersect. 

Lift this ideal decomposition $\mathcal T$ to a tiling of $\HH^3$, and lift the cut locus $\Sigma$ as well. As in Theorem \ref{Thm:Canonical}, 
we subdivide each polyhedron $\Delta$ into basins of attraction, with one basin for each ideal vertex. The vertex-transitive symmetry group of each ideal icosidodecahedron implies that the walls of each basin $Q$ are contained in planes of reflective symmetry.

As in Theorem~\ref{Thm:Canonical}, 
we claim that any given $x \in \HH^3 \setminus \widetilde \Sigma$ that has a unique closest horoball $H_x$, this closest horoball is centered at the ideal vertex of the 
basin $Q$ containing $x$. Suppose for contradiction that $H_x$ is not centered at the ideal vertex of $Q$. Then the geodesic from $x$ to $H_x$, called $\gamma$, must intersect some reflective plane $\Pi$ that contains a wall of $Q$. Then, just as in Theorem~\ref{Thm:Canonical}, we may reflect $\gamma$ across $\Pi$ to find a closer horoball, obtaining a contradiction.
\end{proof}

\begin{remark}\label{Rem:S^3Canonical}
The tiling links on $S^2$ 
  corresponding to the $[3,3,3,3]$ and $[4,3,4,3]$ tilings also share the property that their checkerboard polyhedral decompositions are their canonical decompositions. This can be shown by repeating  the steps of the proof of Proposition~\ref{Prop:S^3Arithmeticity}, apart from the paragraph about arithmeticity, substituting $m=n=3$ or $m=4$ and $n=3$. 

These links and their polyhedral decompositions have been extensively studied. The $[3,3,3,3]$ link is the Borromean rings, and its hyperbolic structure (consisting of two regular octahedra) is constructed in Thurston's notes \cite[Chapter 3]{Thurston:Notes}. The $[4,3,4,3]$ is the cuboctahedral link, and its hyperbolic structure (consisting of two right-angled cuboctahedra) is described in Hatcher \cite[Section 3, Example 1]{Hatcher:Arithmetic}.
\end{remark}

\subsection{Proofs of the main theorems}

We can now assemble all of the above results to prove Theorems~\ref{Thm:MainArithmeticity} and~\ref{Thm:MainCommensurability}, which were stated in the introduction.

\begin{named}{Theorem~\ref{Thm:MainArithmeticity}}
    The following right-angled tiling links have arithmetic exteriors:
    \begin{itemize}
\item The links on $S^2$ corresponding to the $[3,3,3,3]$ and $[4,3,4,3]$ tilings.
    
\item The links on $T^2$ corresponding to the $[4,4,4,4]$ and $[6,3,6,3]$ tilings.
         
\item The links on $S_g$ for $g>1$ corresponding to the $[6,4,6,4]$ and $[6,6,6,6]$ tilings.
    \end{itemize}
    All other right-angled tiling links are non-arithmetic.
\end{named}

\begin{proof}
Let $L$ be a right-angled tiling link corresponding to a $[m,n,m,n]$ tiling of a surface $F$. We proceed based on the genus of $F$.

The only right-angled tiling links on $S^2$ come from the $[3,3,3,3]$ tiling, the $[4,3,4,3]$ tiling, and the $[5,3,5,3]$ tiling. 
By Hatcher's theorem~\cite{Hatcher:Arithmetic}
the $[3,3,3,3]$ and $[4,3,4,3]$ tiling links are arithmetic. 
% (In particular, the $[3,3,3,3]$ link is the Borromean rings.) 
Meanwhile, by Proposition~\ref{Prop:S^3Arithmeticity}, the $[5,3,5,3]$ tiling link is non-arithmetic.

The only right-angled tiling links on $T^2$ correspond to the $[4,4,4,4]$ and $[6,3,6,3]$ tilings. By a theorem of Champanerkar, Kofman, and Purcell~\cite[Theorem 4.1]{CKP:Biperiodic}, both of these links are arithmetic.

Finally, Theorem~\ref{Thm:TilingArithmeticity} shows that the only arithmetic tiling links on hyperbolic surfaces correspond to the $[6,4,6,4]$ and $[6,6,6,6]$ tilings.
\end{proof}

\begin{named}{Theorem~\ref{Thm:MainCommensurability}}
Suppose that $L \subset F \times I$ and $L' \subset F' \times I$ are right-angled tiling links. Then their complements (or doubled complements, if the surfaces are hyperbolic) are commensurable if and only if one of the following holds:
    \begin{itemize}
        \item $L$ and $L'$ correspond to the same $[m,n,m,n]$ tiling.
        \item Both $L$ and $L'$ correspond to  $[3,3,3,3]$, $[4,4,4,4]$, or $[6,6,6,6]$ tilings.
    \end{itemize}
\end{named}

\begin{proof}
For this proof, let $M$ denote either $\Ext(L)$ if $L$ is a tiling link on $S^2$ or $T^2$, or the double $D\Ext(L)$ if $L$ is a tiling link on a hyperbolic surface. We adopt the same convention for $M'$.

Suppose that $L$ and $L'$ correspond to the same $[m,n,m,n]$ tiling. If the surfaces $F$ and $F'$ are hyperbolic, then 
Corollary \ref{Cor:SameTileCommens} implies $M$ and $M'$ are commensurable. Champanerkar, Kofman, and Purcell showed the corresponding statement for the torus case~\cite[Theorem 4.1]{CKP:Biperiodic}. In the $S^3$ case, Gan showed that there is only one right-angled tiling link corresponding to each tiling \cite{Gan:Alternating}.  

If both $L$ and $L'$ correspond to $[3,3,3,3]$, $[4,4,4,4]$, or $[6,6,6,6]$ tilings, then $M$ and $M'$ are arithmetic  with invariant trace field $\Q(i)$. This is proved in  Hatcher~\cite{Hatcher:Arithmetic}, Champanerkar, Kofman, and Purcell~\cite[Theorem 4.1]{CKP:Biperiodic}, and Corollary \ref{Cor:InvariantTraceFields}, respectively.  All of these manifolds are commensurable to the Bianchi orbifold $\HH^3 / PSL(2, \Z[i])$, and therefore to one another.

Next, suppose that $M$ and $M'$ are commensurable and non-arithmetic. By Theorem~\ref{Thm:MainArithmeticity}, either both of them correspond to hyperbolic tilings, or one of them corresponds to the $[5,3,5,3]$ tiling of $S^2$. If $M$ and $M'$ correspond to hyperbolic tilings, Theorem~\ref{Thm:HighGenusCommensurability} implies the tiling types must be equal. Now suppose without loss of generality that $M$ is the complement of the $[5,3,5,3]$ link. Combining Theorem~\ref{Thm:GHH+Margulis} with Theorem~\ref{Thm:Canonical} and Proposition~\ref{Prop:S^3Arithmeticity}, we see that $M'$ is the complement of the $[5,3,5,3]$ link as well.

Finally, suppose that $M$ and $M'$ are commensurable and arithmetic. Then they must have the same invariant trace field \cite[Theorem 3.34]{MR:Arithmetic}. Theorem~\ref{Thm:MainArithmeticity} provides a complete list of the associated tilings. The tiling type of $L$ determines the invariant trace field of $M$, as follows:
\begin{itemize}
\item If $L$ corresponds to the $[3,3,3,3]$, $[4,4,4,4]$, or $[6,6,6,6]$ tiling, then as already discussed $M$ has invariant trace field $\Q(i)$. 
\item If $L$ corresponds to the $[4,3,4,3]$ tiling, then $M$ has invariant trace field $\Q(i\sqrt{2})$, by~\cite{Hatcher:Arithmetic}.
\item If $L$ corresponds to the $[6,3,6,3]$ tiling, then $M$ has invariant trace field $\Q(i\sqrt{3})$, by \cite[Theorem 4.1]{CKP:Biperiodic}.
\item If $L$ corresponds to the $[6,4,6,4]$ tiling, then $M$ has invariant trace field $\Q(i\sqrt{6})$ by Corollary~ \ref{Cor:InvariantTraceFields}.
\end{itemize}
The same applies to $M'$. In particular, the only invariant trace field with more than one tiling type is $\Q(i)$.
\end{proof}

   \bibliographystyle{amsplain}
\bibliography{biblio.bib}

\end{document}